\documentclass{amsart}

\usepackage[latin1]{inputenc}
\usepackage{subfigure}
\usepackage{color}
\usepackage{amsmath, latexsym, amssymb,longtable}
\usepackage{graphicx}
\usepackage{amsfonts}
\usepackage{amsmath}
\usepackage{graphics}
\usepackage{amssymb}
\usepackage{latexsym}
\usepackage{amscd}
\usepackage{color}
\usepackage{pict2e}
\usepackage[all]{xy}
\usepackage{bm}
\usepackage{enumerate}
\usepackage{todonotes}
\usepackage{mathrsfs}

\newtheorem{theorem}{Theorem}[section]
\newtheorem{proposition}[theorem]{Proposition}

\newtheorem{lemma}[theorem]{Lemma}
\newtheorem{cor}[theorem]{Corollary}
\newtheorem{procedure}[theorem]{Procedure}
\newtheorem{conjecture}[theorem]{Conjecture}
\newtheorem{definition}[theorem]{Definition}
\newtheorem{example}[theorem]{Example}

\numberwithin{equation}{section}
\numberwithin{figure}{section}
\numberwithin{table}{section}

\newlength\cellsize 
\savebox2{%
\begin{picture}(12,12)
\put(0,0){\line(1,0){12}}
\put(0,0){\line(0,1){12}}
\put(12,0){\line(0,1){12}}
\put(0,12){\line(1,0){12}}
\end{picture}}
\newcommand\cellify[1]{\def\thearg{#1}\def\nothing{}%
\ifx\thearg\nothing
\vrule width0pt height\cellsize depth0pt\else
\hbox to 0pt{\usebox2\hss}\fi%
\vbox to 12\unitlength{
\vss
\hbox to 12\unitlength{\hss$#1$\hss}
\vss}}
\newcommand\tableau[1]{\vtop{\let\\=\cr
\setlength\baselineskip{-16000pt}
\setlength\lineskiplimit{16000pt}
\setlength\lineskip{0pt}
\halign{&\cellify{##}\cr#1\crcr}}}
\savebox3{%
\begin{picture}(12,12)
\put(0,0){\line(1,0){12}}
\put(0,0){\line(0,1){12}}
\put(12,0){\line(0,1){12}}
\put(0,12){\line(1,0){12}}
\end{picture}}
\newcommand\expath[1]{%
\hbox to 0pt{\usebox3\hss}%
\vbox to 12\unitlength{
\vss
\hbox to 12\unitlength{\hss$#1$\hss}
\vss}}

\newcommand{\geee}{g}
\newcommand{\geea}{g}
\newcommand{\geed}{g}
\newcommand{\geef}{g}
\newcommand{\geeg}{k}
\newcommand{\geeh}{g}
\newcommand{\dee}{a}
\newcommand{\cee}{p}
\newcommand{\mTp}{(T',m)}

\newcommand{\QS}{\check{\mathscr{S}}}
\newcommand{\YQS}{\hat{\mathscr{S}}}
\newcommand{\dI}{\mathfrak{S}^*}
\newcommand{\I}{\mathfrak{S}}
\newcommand{\dYQS}{\hat{\bf{s}}}
\newcommand{\ree}[1]{(\ref{#1})}

\author{Edward E. Allen \and Joshua Hallam \and Sarah K. Mason}
\title[Dual Immaculates as Young Quasisymmetric Schurs]{Dual Immaculate Quasisymmetric Functions Expand Positively into Young Quasisymmetric Schur Functions}
\address{Department of Mathematics, Wake Forest University}
\keywords{quasisymmetric functions, dual immaculate functions, Schensted insertion, Schur functions, tableaux}

\begin{document}
\maketitle

 \paragraph{Abstract.}
We describe a combinatorial formula for the coefficients when the dual immaculate quasisymmetric functions are decomposed into Young quasisymmetric Schur functions.  We prove this using an analogue of Schensted insertion.  Using this result, we give necessary and sufficient conditions for a dual immaculate quasisymmetric function to be symmetric.  Moreover, we show that the product of a Schur function and a dual immaculate quasisymmetric function expands positively in the Young quasisymmetric Schur basis.  We also discuss the decomposition of the Young noncommutative Schur functions into the immaculate functions.  Finally, we provide a Remmel-Whitney-style rule to generate the coefficients of the decomposition of the dual immaculates into the Young quasisymmetric Schurs algorithmically and an analogous rule for the decomposition of the dual bases.

\section{Introduction}\label{sec:intro}

The Schur functions are a fundamental object of study in the areas of algebraic combinatorics, representation theory, and geometry.  They were introduced by Cauchy in 1815~\cite{Cau1815} and appeared in Schur's seminal dissertation~\cite{Sch73} as the characters of the irreducible representations of the general linear group $GL(n,\mathbb{C})$.  Schur functions can be generated by means of divided difference operators, raising operators, matrix determinants, and monomial weights.  (See texts such as \cite{Ful97,Mac08,Sag01,Sta99} for details.)  The multiplication of Schur functions is equivalent to the Schubert calculus on intersections of subspaces of a vector space~\cite{Sta77}.  The Schur functions form an orthonormal basis for the graded Hopf algebra $Sym$ of symmetric functions~\cite{Gei77}.   Symmetric functions appear in classical invariant theory results such as the Chevalley-Shephard-Todd Theorem~\cite{Che55, SheTod54} as well as more recent developments such as the theory of Macdonald polynomials~\cite{Mac88}, nonsymmetric Macdonald polynomials~\cite{Mar99}, and their related combinatorics~\cite{GarRem05,Hag04a, Hag06}.  The algebra $Sym$ of symmetric functions generalizes to both a nonsymmetric analogue $QSym$ and a noncommutative analogue $NSym$.

Stanley laid the foundation for the algebra $QSym$ of quasisymmetric functions through his work on $P$-partitions~\cite{Sta72}.   Gessel ~\cite{Ges84} formalized the definition of quasisymmetric functions and introduced the fundamental basis.  Ehrenborg~\cite{Ehr96} further developed the Hopf algebra structure of $QSym$, which is the Hopf algebra dual to the noncommutative symmetric functions $NSym$.   $QSym$ also plays an important role in permutation enumeration~\cite{GesReu93} and reduced decompositions for finite Coxeter groups~\cite{Sta84}.  Quasisymmetric functions appear in probability theory through the study of random walks~\cite{HerHsi09} and riffle shuffles~\cite{Sta01}.  They also arise in representation theory as representations of Lie algebras~\cite{GesReu93} and general linear Lie superalgebras~\cite{Kwo09} and in the study of Hecke algebras~\cite{Hiv00}.  Discrete geometers use quasisymmetric functions in the study of the {\bf cd}-index~\cite{BHvW06} and as flags in graded posets~\cite{Ehr96}.    Quasisymmetric functions are ubiquitous in combinatorics in part because $QSym$ is the terminal object in the category of combinatorial Hopf algebras~\cite{ABS06}.

In~\cite{HLMvW09}, Haglund \textit{et al.} introduced a new basis for quasisymmetric functions
called the \emph{quasisymmetric Schur functions} $\{\QS_\gamma\}_\gamma$.
The quasisymmetric Schur functions are specializations of nonsymmetric Macdonald polynomials obtained by setting $q=t=0$ in the combinatorial formula described in~\cite{HHL08} and summing the resulting Demazure atoms over all weak compositions which collapse to the same strong composition.  This new basis satisfies many properties similar to those enjoyed by the Schur functions including a Robinson-Schensted-Knuth style bijection with matrices~\cite{HLMvW10}, a Pieri-style multiplication rule~\cite{HLMvW09}, and an omega operation~\cite{MasRem10}.  Haglund \textit{et al.} ~\cite{HLMvW10} provide a refinement of the Littlewood-Richardson rule which gives a formula for the coefficients appearing in the product of a quasisymmetric Schur function and a Schur function when expanded in terms of the quasisymmetric Schur function basis.  The quasisymmetric Schur functions are generated by fillings of composition diagrams analogously to how Schur functions are generated by semistandard Young tableaux.  In representation theory, quasisymmetric Schur functions are dual to noncommutative irreducible characters of the symmetric group~\cite{vWi13}.  The {\it Young quasisymmetric Schur functions}~\cite{LMvW13} are variants of quasisymmetric Schur functions obtained by reversing the entries in composition diagrams.  In this paper, we work with the Young quasisymmetric Schur functions.

The algebra $NSym$ of noncommutative symmetric functions plays an important role in representation theory due to its relationship to quantum linear groups, Hecke algebras at $q=0$~\cite{KroThi97}, and the universal enveloping algebra of $gl_N$~\cite{KroThi99}.  In addition, $NSym$ is isomorphic to the Solomon descent algebra~\cite{GKLLRT95,MalReu95}.

The {\it immaculate basis} for $NSym$, introduced in~\cite{BerBerSalSerZab14}, is constructed using non-commutative Bernstein operators.  The immaculate basis appears in representation theory in relation to indecomposable modules of the $0$-Hecke algebra~\cite{BBSSZ15}.  The forgetful map projects the immaculate basis onto the Schur basis and there exists a Jacobi-Trudi-style formula for constructing the immaculate basis~\cite{BerBerSalSerZab14}.  The {\it dual immaculate quasisymmetric functions} form the dual to the immaculate basis.  Like the quasisymmetric Schur functions, they are generated using fillings of composition diagrams and form another quasisymmetric analogue to Schur functions.  This basis has a positive expansion in terms of the monomial and fundamental bases of $QSym$~\cite{BerBerSalSerZab14} as well as a Pieri rule~\cite{BerSanZab14,BBSSZ13}.

In this paper, we investigate the connection between these two quasisymmetric analogues of Schur functions.  In particular, we show that the dual immaculate basis, $\{{\dI}_\alpha\}_\alpha$,
decomposes as a nonnegative sum of Young quasisymmetric Schur functions $\{\YQS_\gamma\}_\gamma.$

\begin{theorem}{\label{thm:main}}
The dual immaculate quasisymmetric functions decompose into Young quasisymmetric Schur functions in the following way:
$$
\dI_{\alpha} = \sum_{\beta} c_{\alpha,\beta} \YQS_{\beta}
$$
where $c_{\alpha,\beta}$ is the number of DIRTs (Definition~\ref{dirtDef}) of shape $\beta$ with row strip shape $\alpha^{rev}$ (Definition~\ref{def:rowstrip}).
\end{theorem}

This result and its dual version, Theorem~\ref{thm:mainDual}, describe the relationship between two very different quasisymmetric analogues of Schur functions as well as their dual bases in $NSym$ whose connection is not apparent from their definitions.  Dual immaculate quasisymmetric functions decompose positively into Young quasisymmetric Schurs, which then further decompose positively into Gessel's fundamental quasisymmetric functions, creating a tower of Schur-like objects.  The proof of Theorem~\ref{thm:main} involves a Schensted-like insertion algorithm.  The coefficients appearing in this decomposition can be obtained through a combinatorial algorithm similar to the Remmel-Whitney approach to computing Littlewood-Richardson coefficients.  We use this Theorem to obtain new proofs of several results about dual immaculate quasisymmetric functions.  In particular, since any symmetric function which is quasisymmetric Schur positive must be Schur positive, dual immaculate positivity of a symmetric function implies Schur positivity.  We also prove that a dual immaculate quasisymmetric function is symmetric if and only if it is indexed by a certain type of hook shape.  Finally, we show that the product of a Schur function and a dual immaculate quasisymmetric function expands positively into the Young quasisymmetric Schur basis.  

The remainder of the paper is organized as follows.  In Section~\ref{backgroundSec}, we review the background material on compositions and their diagrams.  We then define the Young quasisymmetric Schur functions as well as the dual immaculate quasisymmetric functions and explain their decompositions in the fundamental basis.  Section~\ref{Sec:insertion} describes the insertion algorithm that is used to prove our main result.  We then discuss the proof of our main theorem in Section~\ref{mainThmSec}.  This section also includes some results about the properties of dual immaculate recording tableaux, the connections with $Sym$, and  the decomposition of the dual bases in $NSym$. 
In Section~\ref{sec:RWAlg}, we provide Remmel-Whitney-style algorithms that compute the coefficients of the decomposition in $QSym$ and $NSym$.  We  conclude with a section on future directions.

\section{Background}\label{backgroundSec}

A {\it composition} $\alpha$ of $n$, written $\alpha \vDash n$, is a finite sequence of positive integers that sum to $n$.  If $\alpha=(\alpha_1, \alpha_2, \hdots , \alpha_\ell)$, then $\alpha_i$ is the $i^{th}$ {\it part} of $\alpha$ and $\ell(\alpha)=\ell$ is the \emph{length} of $\alpha$.  If $\alpha=(\alpha_1,\alpha_2,\dots, \alpha_\ell)$ then we define the reverse of $\alpha$ to be $\alpha^{rev} = (\alpha_\ell,\alpha_{\ell-1},\dots, \alpha_1)$.   A composition $\beta$ is said to be a {\it refinement} of a composition $\alpha$ if $\alpha$ can be obtained from $\beta$ by summing collections of consecutive parts of $\beta$.  We say that a composition $\beta$ is a \emph{rearrangement} of a composition $\alpha$ if the parts of $\beta$ can be reordered to form $\alpha$.  For example, $(3,4,1,3)$ is a rearrangement of $(1,4,3,3)$.    Given two compositions $\alpha=(\alpha_1,\alpha_2,\dots, a_\ell)$ and $\beta=(\beta_1,\beta_2,\dots, \beta_k)$ we say $\alpha \succeq \beta$ in \emph{dominance order} if  $\alpha_1+\alpha_2+\cdots \alpha_i \geq \beta_1+\beta_2+\cdots+ \beta_i$ for all $i\geq 1$.  Here we make the assumption that if $i> \ell$ then $\alpha_i=0$ and if $i>k$ then $\beta_i=0$.    A composition is a {\it partition} if $\alpha_i \ge \alpha_{i+1}$ for all $1 \le i \le \ell-1$. Finally, if $\alpha$ is  a composition then we  define $set(\alpha) =  \{\alpha_1,\alpha_1+\alpha_2,\dots, \alpha_1+\alpha_2+\dots + \alpha_{\ell-1}\}$.

Given a composition $\alpha=(\alpha_1, \alpha_2, \hdots , \alpha_\ell)$, the {\it diagram} $D_\alpha$ is constructed by placing boxes (or {\it cells}) into left-justified rows so that the $i^{th}$ row from the bottom contains $\alpha_i$ cells.  The shape of $D_\alpha$ is denoted by $\alpha$.  This is analogous to the French notation for the Young diagram of a partition.  Position $(i,j)$ in $D_\alpha$ refers to the cell in the $i^{th}$ column (reading from left to right) and the $j^{th}$ row (reading from bottom to top).  For example, the diagram $D_\alpha$ pictured below corresponds to a diagram of shape $\alpha=(2,4,3)$ with an X in position $(3,2)$.
$$
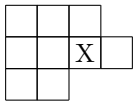
\tableau{ {}&{}&{}\\ {}&{}&{\mbox{X}} &{} \\{}&{}}\label{E:Tableau}
$$

A {\it quasisymmetric function} is a bounded degree formal power series $f(x) \in \mathbb{Q}[[x_1, x_2, \hdots ]]$ such that for all compositions $\alpha=(\alpha_1, \alpha_2, \hdots, \alpha_\ell)$, the coefficient of $\prod x_i^{\alpha_i}$ is equal to the coefficient of $\prod x_{i_j}^{\alpha_i}$ for all $i_1 < i_2 < \cdots < i_\ell$.  Let $QSym$ denote the algebra of quasisymmetric functions
and $QSym_n$ denote the space of homogeneous quasisymmetric functions
of degree $n$, so that $$QSym = \bigoplus_{n \geq 0} QSym_n.$$

A natural basis for $QSym_n$  is the {\it monomial quasisymmetric basis}, given by the collection $\{M_\alpha\}_{\alpha \vDash n}$ where
$$\displaystyle{M_{\alpha} = \sum_{i_1 < i_2 < \cdots < i_\ell} x_{i_1}^{\alpha_1} x_{i_2}^{\alpha_2} \cdots x_{i_\ell}^{\alpha_\ell}.}$$   Gessel's {\it fundamental basis for quasisymmetric functions}~\cite{Ges84} can be expressed by $$\displaystyle{F_{\alpha} = \sum_{\beta } M_{\beta}},$$ where  the sum is over all $\beta$ which are refinements of $\alpha$.

Given a diagram $D_{\alpha}$, a \emph{filling of $D_\alpha$} is a function $G:D_\alpha \rightarrow \mathbb{Z}_+$.  Here $G(i,j)$ denotes the image of the cell $(i,j)$ and is called the \emph{entry} of cell $(i,j).$

\begin{definition}\label{SSYRTdef}~\cite{LMvW13}
The filling $T: D_\alpha  \rightarrow \mathbb{Z}_+$ is a \emph{semistandard Young composition tableau (SSYCT)} of shape $\alpha$ if it satisfies the following conditions:
\begin{enumerate}
\item Row entries are weakly increasing from left to right (i.e., $T(i,j)\le T(i+1,j)$ for all $(i,j), (i+1,j) \in D_\alpha$).
\item The entries in the leftmost column are strictly increasing from bottom to top (i.e., $T(1,j)<T(1,j+1)$ for all $(1,j), (1,j+1) \in D_\alpha$).
\item (Young composition triple rule)
For all $\{i,j,k\}$ such that $1\leq j<k \leq \ell(\alpha)$ and $1\leq i<max \{\alpha_j,\alpha_k\},$ if ${T}(i,k) \le { T}(i+1,j)$, then ${T}(i+1,k)< { T}(i+1,j)$
under the assumption that the entry in any cell not contained in $D_\alpha$ is $\infty$.
\end{enumerate}
\end{definition}

\noindent
Less formally, the Young composition triple rule states that for any subarray in ${T}$ (shown in Figure \ref{YCTTripRuleFig}), if $b \leq a$, then $c<a$.
Here we assume that if the position immediately right of $b$ is empty, then $c=\infty$.  Additionally, we set the \emph{augmentation of $T$}, denoted by $\bar{T}$, to be the filling of $D_{\bar{\alpha}}$, where
$\bar{\alpha}=(\alpha_1+1, \alpha_2+1, \ldots, \alpha_\ell+1),$ in which
the right-most entry in each row is $\infty$ and the remaining cells have the same filling as $T$.  (Here we abuse notation by allowing infinities in our augmentation, while technically infinities are not allowed to be entries in a filling.)  We consider $\bar{T}$ to be a SSYCT if $T$ is a SSYCT.
See Figure ~\ref{fig:ct-reading} for an example.

\begin{figure}
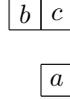

\begin{center}
$\tableau{b&c}$

\vspace{12 pt}

$\hspace{12 pt} \tableau{a}$

\end{center}

\caption{Young composition triple rule:  $b \leq a \Rightarrow c<a$}
\label{YCTTripRuleFig}
\end{figure}

\noindent

The following definition will be useful in Section~\ref{Sec:insertion}.

\begin{definition}
Read the entries in the columns of a Young composition tableau $T$ (or its augmentation $\bar{T}$) from top to bottom, beginning with the rightmost column of $T$ and working right to left.  This ordering of the cells is called the \emph{Young reading order}.  When the entries of the cells are read in Young reading order, the resulting word is called the \emph{Young reading word} of $T$, denoted $rw_{\YQS}(T)$.  See Figure ~\ref{fig:ct-reading} for an example.
\end{definition}
\noindent
Note that we will also define an \emph{immaculate reading word} in Definition \ref{ITRWdef}.

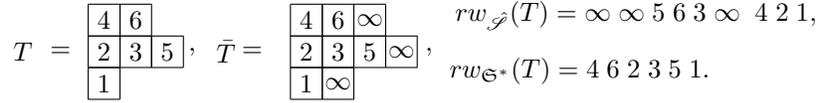
\begin{figure}
\begin{center}
\begin{tikzpicture}
\node at (-.2,0) {$T$};
\node at (.3,0) {$=$};
\node at (1.3,0) {$\tableau{4 & 6 \\ 2 & 3 & 5 \\ 1 }$};
\node at (2.05,0) {$,$};
\node at (2.5, 0) {$ \bar{T}$};
\node at (2.85,0)  {$=$};
\node at (4.2,0) {$ \tableau{4 & 6 & \infty \\ 2 & 3 & 5 & \infty \\ 1 &  \infty}$};
\node at (5.2,0) {$,$};
\node at (7.95,.5) {$ rw_{\YQS}(\bar{T}) = \infty \; \infty \; 5 \; 6 \; 3 \; \infty \;  \; 4 \; 2 \; 1,$};
\node at (7.2,.5-.75) {$ rw_{\dI}(T) = 4\; 6\; 2\;3\;5\;1.$};
\end{tikzpicture}
\caption{As an augmentation of a  Young composition tableau $T$, the Young reading word  of $\bar{T}$ is $\infty \; \infty \; 5 \; 6 \; 3 \; \infty \;  \; 4 \; 2 \; 1$.  As an immaculate tableau, the immaculate reading word of $T$ is  $4\; 6\; 2\;3\;5\;1$.\label{fig:ct-reading}}
\label{fig:ct-reading}
\end{center}
\end{figure}

The {\em weight} of a SSYCT $T$ of shape $\alpha$ is the monomial $\displaystyle{x^T=\prod_i x_i^{v_i}}$ where $v_i$ is the number of times the entry $i$ appears in $T$ as seen in Figure~\ref{SSYCTweights}.  A \emph{standard Young composition tableau} (SYCT) of shape $\alpha \vDash n$ is a semistandard Young composition tableau in which each of the numbers $\{1, \hdots , n\}$ appears exactly once.

\begin{definition}{\cite{LMvW13}}
Let $\alpha$ be a composition.  Then the Young quasisymmetric Schur function $\YQS_{\alpha}$ is given by $$\YQS_{\alpha}=\sum_T x^T,$$ summed over all semistandard Young composition tableaux $T$ of shape $\alpha$.  See Figure~\ref{SSYCTweights} for an example.
\end{definition}

We now describe the method given in Proposition 5.2.2 of \cite{LMvW13} for writing a Young quasisymmetric Schur function as a positive sum of Gessel's fundamental quasisymmetric functions.

\begin{definition}
The \emph{Young descent set, $Des_{\YQS}(T)$}, of a standard Young composition tableau $T$ is the subset of $\{1, \hdots, n-1\}$ consisting of all entries $i$ of $T$ such that $i+1$ appears weakly to the left of $i$ in $T$.
\end{definition}
We use a subscript to denote the Young descent set $Des_{\YQS}(T)$, which is not usually done.  We do this because we will use another type of descent set, the
 \emph{immaculate descent set}, later in this development; see Definition~\ref{DefDesSIT}.

\begin{proposition}{\label{fundDecomp}}{~\cite{LMvW13}}
Let $\alpha, \beta$ be compositions.  Then $$\YQS_{\alpha} = \sum_{\beta} d_{\alpha, \beta} F_{\beta},$$ where $d_{\alpha, \beta}$ is equal to the number of standard Young composition tableaux $T$ of shape $\alpha$ such that $Des_{\YQS}(T) =  set(\beta)$.
\end{proposition}

The example in Figure~\ref{SSYCTweights} shows that there is only one SYCT of shape $(1,2,1)$. It has Young descent set $\{1,3\}$ and therefore $\YQS_{(1,2,1)} = F_{(1,2,1)}$.

In~\cite{BerBerSalSerZab14}, the authors introduce a new basis of NSym called the immaculate basis.   Since QSym and NSym are dual, this gives rise to a dual  basis of QSym called the \emph{dual immaculate basis}.  One can define the dual immaculate quasisymmetric functions using {\it immaculate tableaux}.

\begin{figure}
$$\tableau{3 \\ 2 & 2 \\ 1} \qquad \tableau{4 \\ 2 & 2 \\ 1} \qquad \tableau{4 \\ 2 & 3 \\ 1} \qquad \tableau{4 \\ 3 & 3 \\ 1} \qquad \tableau{4 \\ 3 & 3 \\ 2} $$
$$\YQS_{121}(x_1,x_2,x_3,x_4)=x_1x_2^2x_3 + x_1x_2^2x_4 + x_1x_2x_3x_4 + x_1x_3^2x_4 + x_2x_3^2x_4$$
\caption{The SSYCT that generate $\YQS_{(1,2,1)}(x_1,x_2,x_3,x_4)$.}
\label{SSYCTweights}
\end{figure}

\begin{definition}~\cite{BerBerSalSerZab14}\label{SSYRTdef}
A filling $U: D_{\alpha}  \rightarrow \mathbb{Z}_+$ is an \emph{immaculate tableau} of shape $\alpha$ if it satisfies the following conditions:

\begin{enumerate}
\item Row entries are weakly increasing from left to right (i.e., $U(i,j)\le U(i+1,j)$ for all $(i,j), (i+1,j) \in D_\alpha$).
\item The entries in the leftmost column are strictly increasing from bottom to top (i.e., $U(1,j)<U(1,j+1)$ for all $(1,j), (1,j+1) \in D_\alpha$).
\end{enumerate}
\end{definition}

Note that the dual immaculate basis was originally introduced using English notation.  In the above definition we use the French notation for our tableaux; this is why in condition (2) above, the entries in the leftmost column increase from bottom to top rather than top to bottom.  Our definition of immaculate descent also reflects this modification.  We use French notation rather than English simply to preserve compatibility with the Young composition tableaux; none of the underlying mathematics is impacted in any way by this cosmetic convention.

Observe that every SSYCT is also an immaculate tableau since the definition is the same except that immaculate tableaux are not required to satisfy the Young composition triple rule.
We will now define the \emph{immaculate reading word} $rw_{\dI}(U)$ for an immaculate  tableau $U$.  Note that it is not the same as the Young reading word $rw_{\YQS}(U)$ for the Young composition tableau $U$.

\begin{definition} \label{ITRWdef}
Read the entries in the rows of an immaculate tableau $U$, from left to right,  beginning with the highest row of $U$ and working top to bottom.  The resulting word is called the \emph{immaculate reading word} of $U$, denoted $rw_{\dI}(U)$.  See Figure~\ref{fig:ct-reading} for an example.
\end{definition}

  We say that an immaculate tableau is a \emph{standard immaculate tableau} if the numbers $\{1,\dots, n\}$ each appear exactly once.  Just as with a Young composition tableau, the {\em weight} of an immaculate tableau $U$ of shape $\alpha$ is the monomial $x^U=\prod_i x_i^{v_i}$, where $v_i$ is the number of times the entry $i$ appears in $U$ as seen in Figure~\ref{ImTabWeightsFig}.

\begin{definition}
Let $\alpha$ be a composition.  The \emph{dual immaculate quasisymmetric function} $\dI_{\alpha}$ is given by $$\dI_{\alpha}=\sum_U x^U,$$ where the sum is over all immaculate tableaux of shape $\alpha$.  See Figure~\ref{ImTabWeightsFig} for an example.
\end{definition}

\begin{figure}
\begin{center}
\begin{tikzpicture}
\node at (0,0) {\tableau{3 \\ 2 & 2 \\ 1}};
\node at (1.25,0){\tableau{4 \\ 2 & 2 \\ 1}};
\node at (2.5, 0) { \tableau{4 \\ 2 & 3 \\ 1}};
\node at (3.75, 0) {\tableau{4 \\ 3 & 3 \\ 1}};
\node at (5, 0) { \tableau{4 \\ 3 & 3 \\ 2} };
\node at (6.25,0) { \tableau{3\\2&3\\1}};
\node at (7.5,0) {\tableau{3\\2&4\\1}};
\node at (2.5,-1.75) {\tableau{4\\2&4\\1}};
\node at (3.75, -1.75) { \tableau{4\\ 3&4\\2}};
\node at (5, -1.75)  { \tableau{4 & \\ 3 & 4 \\ 1}};

\end{tikzpicture}
\begin{align*} \dI_{121}(x_1,x_2,x_3,x_4)  =  &\  x_1x_2^2x_3 + x_1x_2^2x_4 + 2x_1x_2x_3x_4 + x_1x_3^2x_4 + x_2x_3^2x_4+x_1x_2x_3^2\\&+x_1x_2x_4^2
  +  x_2x_3x_4^2+x_1x_3x_4^2
\end{align*}
\caption{The immaculate tableaux that generate $\dI_{(1,2,1)}(x_1,x_2,x_3,x_4)$.}
\label{ImTabWeightsFig}
\end{center}
\end{figure}

Just as Young quasisymmetric Schur functions decompose into positive sums of fundamental quasisymmetric functions, the dual immaculate quasisymmetric functions decompose into the fundamental basis using descent sets.  Now we define the immaculate descent set of a standard immaculate tableau.

\begin{definition}
The \emph{immaculate descent set, $Des_{\dI}(U)$}, of a standard immaculate tableau $U$ is the subset of $\{1, \ldots, n-1\}$ consisting of all entries $i$ of $U$ such that $i+1$ appears strictly above $i$ in $U$.\label{DefDesSIT}
\end{definition}

 As an example, consider the filling
\begin{center}
\begin{tikzpicture}
\node at (0,0) {$G=$};
\node at (1,0) {$\tableau{2&3\\ 1&4}.$};
\end{tikzpicture}
\end{center}
We see that  $Des_{\dI}(G) =\{1\}$.  Note that the immaculate descent set  of a standard immaculate tableau is not the same as the Young descent set of a standard Young composition tableau.   In fact, the tableau $G$ is both a standard immaculate tableau and a standard Young composition tableau.  However, $Des_{\YQS}(G) = \{1,3\}$ and so the two descents sets for the same filling need not be the same.

We now explain how the dual immaculate quasisymmetric functions decompose into the fundamental basis.

\begin{proposition}{\label{fundDecompofDI}}{~\cite{BerBerSalSerZab14}}
Let $\alpha, \beta$ be compositions.  Then $$\dI_{\alpha} = \sum_{\beta} e_{\alpha, \beta} F_{\beta},$$ where $e_{\alpha, \beta}$ equals the number of standard immaculate tableaux $U$ of shape $\alpha$ with $Des_{\dI}(U) =  set(\beta)$.
\end{proposition}

As an example, consider the decomposition of $\dI_{(1,2,1)}$ into the fundamental basis.  Figure~\ref{ImTabWeightsFig} shows that there are two standard immaculate tableaux of shape $(1,2,1)$.  Their immaculate descent sets are $ \{1,3\}$ and $\{1,2\}$.  Thus,  $\dI_{(1,2,1)} = F_{(1,2,1)}+F_{(1,1,2)}$.

\section{An Insertion and Recording Algorithm}{\label{Sec:insertion}}

The insertion procedure $k \rightarrow C$ given in \cite{HLMvW09} maps a positive integer $k$ into a composition tableau $C$.  We describe an analogous procedure (Procedure~\ref{proc:insertion}) that maps a positive integer $k$ into a Young composition tableau $T$ to produce a Young composition tableau $k \rightarrow T$.  Our procedure is equivalent to the procedure for composition tableaux in the sense that applying insertion to a composition tableau and then mapping to a Young composition tableau produces the same result as first mapping to a Young composition tableau and then applying insertion.  Therefore the fact that the procedure $k \rightarrow C$ produces a composition tableau immediately implies that our procedure produces a Young composition tableau.

\begin{procedure}{\label{proc:insertion}}
Let  $(c_1,d_1), (c_2,d_2), \ldots $ be the cells of the augmented diagram $\bar{T}$ listed in Young reading order.  Set $k_0:=k$ and let $i$ be the smallest positive integer such that $\bar{T}(c_i-1,d_i) \le k_0 < \bar{T}(c_i,d_i)$.  If such an $i$ exists, there are two cases.

{\bf Case 1.}  If $\bar{T}(c_i,d_i) =\infty$, then
place $k_0$ in cell $(c_i,d_i)$ and terminate the procedure.

{\bf Case 2.} If $\bar{T}(c_i,d_i) \neq \infty$, then set $k:=\bar{T}(c_i,d_i)$,
place $k_0$ in cell $(c_i,d_i)$, and repeat the procedure by  inserting $k$ into the sequence of cells $(c_{i+1},d_{i+1}), (c_{i+2},d_{i+2}), \ldots$.
In such a situation, we say that $\bar{T}(c_i,d_i)$ is
{\em bumped}.

If no such $i$ exists, begin a new row (containing only $k_0$) in the highest position in the leftmost column such that all entries below $k_0$ in the leftmost column are smaller than $k_0$ and terminate the procedure.  If the new row is not the top row of the diagram, shift all higher rows up by one.
\end{procedure}

The sequence of cells that contain elements which are bumped
in the insertion $k \rightarrow T$ plus the final cell which is added when
the procedure is terminated is called the {\em bumping path}
of the insertion.

\begin{figure}
\begin{center}
\begin{tikzpicture}
\node at (0,0) {$5$};
\node at (.5,0) {$\rightarrow$};
\node at (1.5,0) {$\tableau{6 & 8 \\ 3 & 4 & 7 \\ 2}$};
\node at (2.5,0) {$=$};
\node at (3.5,0) {$ \tableau{6 & {\bf 7} \\ 3 & 4 & {\bf 5} \\ 2 & {\bf 8}}$};
\end{tikzpicture}\caption{The insertion of $5$ into a Young composition tableau of shape $(1,3,2)$.}\label{fig:bumping}
\end{center}
\end{figure}

\begin{figure}
\begin{center}
\begin{tikzpicture}
\node at (0,0) {$2$};
\node at (.5,0) {$\rightarrow$};
\node at (1.5,0) {$\tableau{4 &5 \\1 & 3}$};
\node at (2.5,0) {$=$};
\node at (3.5,0) {$\tableau{4 &5 \\ {\bf 3} \\ 1 & {\bf 2}} $};
\end{tikzpicture}\caption{The insertion of $2$ into a Young composition tableau of shape $(2,2)$ in which a new row is created. }\label{fig:bumping2}
\end{center}
\end{figure}

\begin{example}\label{insertExample}
Let us consider the insertion of $5$ into a Young composition tableau of shape $(1,3,2)$ which is shown in Figure~\ref{fig:bumping}.  The first element bumped is the $7$ in column $3$.  This $7$ is replaced by $5$ and $7$ is then inserted into the remaining sequence of cells.  The $7$ then bumps the $8$ in column $2$ and $8$ is inserted into the remaining cells.  The $8$ is placed to the right of the $2$ and the procedure terminates.  The bumping path is therefore the sequence of cells $\{ (3,2), (2,3), (2,1) \}$.  Notice that the entries
of $\bar{T}$ in the bumping path must strictly increase as we proceed
in Young reading order.

\end{example}

\begin{example}\label{insertExample2}
In the insertion of $2$ into a Young composition tableau of shape $(2,2)$, shown in Figure~\ref{fig:bumping2}, the first element bumped is the $3$ in column $2$.  Since $1 <3 < 4$, a new row is created and the second row is moved up.

\end{example}
Let $U$ be a standard immaculate tableau and let $rw_{\dI}(U)$ be the immaculate reading word of $U$.  We define a procedure that maps $rw_{\dI}(U)$ to a pair $(P,Q)$ consisting of a standard Young composition tableau $P$ and a recording filling $Q$.

Begin with $(P,Q)=(\emptyset, \emptyset)$, where $\emptyset$ is the empty filling.  Let $k_1$ be the first letter in the word $rw_{\dI}(U)=k_1k_2\hdots k_g$.  Insert $k_1$ into $P$ using the insertion procedure described in Procedure~\ref{proc:insertion} and let $P_1$ be the resulting Young composition tableau.  Record the location in $P$ where the new cell was created by placing a ``1'' in $Q$ in the corresponding location and let $Q_1$ be the resulting filling.  Next assume the first $j-1$ letters of $rw_{\dI}(U)$ have been inserted.  Let $k_j$ be the $j^{th}$ letter in $rw_{\dI}(U)$.  Insert $k_j$ into $P_{j-1}$ and let $P_j$ be the resulting diagram.  Place the letter $j$ in the cell of $Q_{j-1}$ corresponding to the new cell in $P_j$ created from this insertion and let $Q_j$ be the resulting filling.

Notice that $P$ is a standard Young composition tableau since the insertion procedure produces a Young composition tableau.  The recording filling $Q$ has the same shape as $P$ by construction, but is not a Young composition tableau.  We now describe the properties of $Q$.   We begin with a definition.

\begin{definition}{\label{def:rowstrip}}
Let $Q$ be a filling of a diagram for $\beta\vDash n$ with the integers $\{1,\dots, n\}$, each appearing exactly once.  A \emph{row strip} of $Q$ is a maximal sequence of consecutive integers, none of which are in the same column of $Q$.  The \emph{row strip shape} of $Q$ is the composition $(\alpha_1,\alpha_2,\dots, \alpha_{\ell})$ where $\alpha_i$ is the length of the row strip sequence which starts with the number $\alpha_1+\alpha_2+\dots +\alpha_{i-1}+1$.
\end{definition}

For an example, consider the filling
\begin{center}
\begin{tikzpicture}
\node at (0,0) {$Q=$};
\node at (1.35,0) {$ \tableau{1&6\\2 &3&4&7\\5}$};
\node at (2.3,-.1) {.};
\end{tikzpicture}
\end{center}

The first row strip is $1$, the next row strip is $2,3,4$, and finally we have $5,6,7$.  It follows that the row strip shape of $Q$ is $(1,3,3)$.

\begin{lemma}
Assume $c \le d$ are inserted into a Young composition tableau $P$ to form $(d \rightarrow (c \rightarrow P))$.  The new cell created by the insertion of $d$ is strictly to the right of the new cell created by the insertion of $c$.  In particular, if a sequence $c_1 \le c_2 \le \cdots \le c_m$ is inserted into a Young composition tableau in order from smallest to largest, then the column indices of the resulting new cells strictly increase.
\end{lemma}

\begin{proof}
We proceed by proving that the bumping path for $d$ terminates in a cell no farther in Young reading order than the cell to the right of the termination point for the insertion of $c$.  In particular, this implies that the insertion procedure for $d$ terminates in a column further to the right than the insertion procedure for $c$.

Let
$(i_1,j_1),(i_2,j_2), \dots, (i_\geee,j_\geee)$
be the bumping path for the insertion $P'=(c \rightarrow P)$.  We assume for now that the insertion of $c$ does not terminate in the leftmost column.
Set $b_0:=c$ and $b_h:=P(i_h,j_h),$ for $1 \le h \le \geee-1.$  In $P$, the cell $(i_\geee,j_\geee)$ is empty but the cell immediately to its left is not.
Note that $b_{h-1}=P'(i_h,j_h)$ for $1 \le h \le \geee-1.$
Since the entries in the rows of $P$ increase from left to right and the reading order goes by columns from right to left,
the insertion algorithm can bump at most one entry in each row.

 Suppose that $d\rightarrow P'$ does not terminate before reaching cell $(i_1+1,j_1)$.  Let $d_0'\ge d$ be the insertion entry as $d\rightarrow P'$ passes through $(i_1+1,j_1)$.
 Since $(i_1,j_1)$ was bumped in $c\rightarrow P$ and at most one entry in each row can be bumped,
 $P(i_1+1,j_1)=P'(i_1+1,j_1).$  If $P'(i_1+1,j_1)$ is bumped by $d_0'$, then since $P'(i_1+1,j_1)$ was immediately to the right of $b_1=P(i_1,j_1)$, we have
$$P'(i_1+1,j_1) =P(i_1+1,j_1) \ge P(i_1,j_1)=b_1.$$
If $P'(i_1+1,j_1)$
 is not bumped, then either $P'(i_1,j_1) > d_0'$ (which can't happen since $P'(i_1,j_1)=c \le d \le d_0'$) or $d_0' \ge P'(i_1+1,j_1)$.  If $d_0' \ge P'(i_1+1,j_1)$, we have
  $$d_0' \ge P'(i_1+1,j_1)=P(i_1+1,j_1) \ge P(i_1,j_1)=b_1.$$
So, in either case, both $d_0'$ and the entry $d_1$ that continues past $(i_1+1,j_1)$
in $d\rightarrow P'$ must be greater than or equal to $b_1=P(i_1,j_1)$.

Recall that
$$b_1=P'(i_2,j_2)<b_2=P(i_2,j_2) \le P(i_2+1,j_2).$$
If the cell $(i_2+1,j_2)$  in $d\rightarrow P'$ is reached, the current insertion entry $d_1'\ge d_1$ must either bump $P'(i_2+1,j_2)$
(which is greater than or equal to $P(i_2,j_2)=b_2$) or
pass through $(i_2+1,j_2)$ without bumping.  Since $d_1' \ge d_1 \ge b_1$, if $d_1'$ doesn't bump $ P'(i_2+1,j_2)$, we must have $$d_1' \ge P'(i_2+1,j_2)=P(i_2+1,j_2)\ge P(i_2,j_2)=b_2.$$
In either case, both $d_1'$ and the entry $d_2$ that continues past $(i_2+1,j_2)$ must be greater than or equal to $b_2$.
Continuing this reasoning implies that if the entry $d_{\geee-1}'$ reaches the cell $(i_\geee+1,j_\geee)$ then
$d_{\geee-1}' \ge b_{\geee-1}=P'(i_\geee,j_\geee)$.  Since the insertion $c\rightarrow P$ terminated,
$b_{\geee-1}$ was inserted at the end of a row of $P$ in location $(i_\geee,j_\geee)$.
Thus, $d_{\geee-1}'$ will be inserted (and the insertion algorithm will terminate) at or before $(i_\geee+1,j_\geee)$.
Note throughout all of this that if the procedure terminates earlier, the conclusion of the lemma is satisfied since that termination will be to the right of the new cell in $P'$.

If the insertion of $c$ terminates in the leftmost column, some of the indices in the bumping path may change but the same argument still applies.

Therefore if $c \le d$, the new cell created during the insertion of $d$ appears strictly to the right of the new cell created during the insertion of $c$.  Repeated application of this argument shows that  if a sequence $c_1 \le c_2 \le \cdots \le c_m$ is inserted into a Young composition tableau in order from smallest to largest, then the column indices of the resulting new cells strictly increase.
\end{proof}

It follows that as we insert a row of a standard immaculate tableau, the column indices of the corresponding elements of the recording filling strictly increase.  In addition, when we start inserting a new row of the standard immaculate tableau, the element we are inserting is the smallest element that has been inserted so far and thus it must placed in a new row at the bottom of the leftmost column.  Therefore the leftmost column of the recording filling must be strictly increasing from top to bottom.  Also, the insertion of a row from a standard immaculate tableau produces a row strip in the recording filling.  The row strip shape obtained from the insertion of a standard immaculate tableau $U$ of shape $\alpha=(\alpha_1, \alpha_2, \hdots , \alpha_{\ell})$ is $\alpha^{rev}$ rather than $\alpha$ since insertion of the immaculate reading word $rw_{\dI} (U)$ begins at the top and hence we insert a row of length $\alpha_\ell$, then $\alpha_{\ell-1}$ and so on.
Therefore, we have the following.

\begin{proposition}{\label{prop:RowStrips}}
 Let $Q$ be any recording filling obtained from inserting the immaculate reading word $rw_{\dI} (U)$ of a standard immaculate tableau $U$ of shape $\alpha$. Then the row strips start in the first (leftmost) column, the first (leftmost) column entries strictly increase from top to bottom, and the row strip shape is $\alpha^{rev}$.
 \end{proposition}

The recording filling we obtain from insertion of a standard immaculate tableau also satisfies a triple rule.  We describe it next.

\begin{definition}
Let $Q$ be a  filling of a diagram $D_\alpha$ for $\alpha\vDash n$ with the integers $\{1,\dots, n\}$.
We say that $Q$ satisfies the \emph{recording triple rule} if whenever $Q(i,j)>Q(i,\geea)$ where $j>\geea$ then $Q(i,j)>Q(i+1,\geea).$  (If cell $(i+1,\geea)$ is empty, we consider it to contain the entry infinity.)  This is equivalent to the statement that if $a>b$ then $a>c$ in the subarray pictured below.  (Here we assume that if the position immediately to the right of $b$ is empty, then $c=\infty$.)
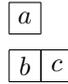
\begin{center}
$\tableau{a}$

\vspace{5 pt}

$\hspace{12 pt} \tableau{b & c}$

\end{center}
\end{definition}

Note that the recording filling $Q$ failing the recording triple rule is equivalent to the situation during the insertion and recording process
in which we have
 the cell $(i,g)$ in $Q$ already filled and we place an entry in position $(i,j)$, where $j>\geea$, before we place an entry in $(i+1,\geea).$
If $(i,j)$ is filled in recording filling $Q$ before we fill $(i+1, \geea)$,
then $Q(i,g)<Q(i,j)<Q(i+1, \geea)$.

\begin{proposition}\label{recTabTripRuleProp}
Let $Q$ be a recording filling obtained by inserting the reading word from a standard immaculate tableau.  Then $Q$ satisfies the recording triple rule.
\end{proposition}

\begin{proof}
The only way $Q$ can fail to satisfy the triple rule is if $b<a<c$ in the triple of cells $\{a,b,c\}$.  To do this, entry $a$ must be placed into the diagram after $b$ but before $c$.  This can only happen if a higher row (the row ultimately containing $a$) has length one less than the lower row containing $b$ just before $a$ is added.  We assume that such an $a$ and $b$ do exist, and argue by contradiction.  Since the leftmost column increases top to bottom, $a$ and $b$ cannot be in the leftmost column in the following.

Consider the partial filling $P$ in Figure~\ref{figForTripRuleProof}.  This figure shows part of a Young composition tableau obtained by inserting part of a reading word of a standard immaculate tableau.  In the figure, suppose that $d=P(i-1,j)$ and $e=P(i,h)$ where $h<j.$
The cells in locations $(i,j)$ and $(i+1,h)$  (immediately to the right of the cells containing $d$ and $e$, respectively) are empty;
hence, we suppose that $P(i,j)=P(i+1,h)=\infty.$   Note that $e<d$; otherwise tableau $P$ would not satisfy the Young
composition triple rule since we would have $P(i-1,j)<P(i,h)$ which implies that $P(i,j)<P(i,h)$ but $P(i,j)=\infty.$ We will show that if we insert another element it cannot be placed in position $(i,j)$.

Suppose that it is possible to place an element in cell $(i,j).$
Let $g$ be the insertion element which passes through the position $(i+1,h)$.
(Such a $g$ exists since otherwise the insertion would end before reaching cell $(i,j)$.)  Since $g$ passes by $(i+1,h)$, we have $g<e$.  Now either $g$ is placed next to $d$ or it bumps something before.  If it is placed next to $d$, then $g>d$.  However, this implies that $e>g>d$ which is a contradiction.  Thus, $g$ must bump some element after passing by $e$.  We break into two cases depending on which column the element which is placed in position $(i,j)$ originates from.

{\bf Case 1.}  Suppose that the element which is placed in position $(i,j)$ originates in column $i+1$.  Then this element must be in a row below the row containing $e$ since $g$ starts to bump in a row below the one containing $e$.   Call this element $f$.  As one can see in Figure~\ref{figForTripRuleProof}, we must have that $f<e$ since otherwise the triple rule for $P$ is not satisfied.  Since $f$ is placed right of $d$, $f>d$.  However, that implies $e>f>d$ which is a contradiction.

{\bf Case 2.} Suppose that the element which is placed in position $(i,j)$ originates in column $i$.   This element must have been bumped.  Let $f_1, f_2,\dots , f_k$ be the elements bumped in column $i$ with $f_k$ being the element which is placed in position $(i,j)$.  Moreover, let $b_i$ be the element immediately to the left of $f_i$ for $1\leq i\leq k$.  Suppose that $f_0$ was the element which bumps $f_1$.  Then either $f_0=g$ or $f_0$ was in column $i+1$ in a row lower than the one containing $e$.  In either case $f_0$ must be smaller than $e$.  Since $f_0$ is placed next to $b_1$, $b_1<f_0<e$.  Thus, the triple rule for $P$ implies that $e>f_1$ since $e>b_1$.  Next, $f_1$ bumps $f_2$.  Thus $e>f_1>b_2$.  Again, the triple rule for $P$ implies that $e>f_2$.  Continuing this reasoning implies that $e>f_k$.  However, this implies that $e>f_k>d$ which is a contradiction.

It follows that when we create the recording filling it is not possible to place an element in a column such that there is already an element in that column below it with nothing immediately to its right.  Therefore the recording filling must satisfy the recording tableau triple rule.
\end{proof}

\begin{figure}
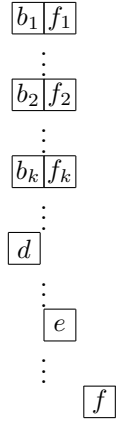


\begin{center}

$\tableau{b_1 & f_1}$

$
\vdots
$

$\tableau{b_2 & f_2 }$

$
\vdots
$

$\tableau{b_k & f_k}$

$
\vdots
$

$\hspace{-15 pt}\tableau{d }$

$
\vdots
$

$\hspace{0 pt}\tableau{&e}$

$
\vdots
$

$\hspace{30 pt}\tableau{&f }$

\end{center}

\caption{Figure from the proof of Proposition~\ref{recTabTripRuleProp}}
\label{figForTripRuleProof}

\end{figure}

\begin{definition}\label{dirtDef}
Let $\beta$ be a composition of $n$.  A filling of a diagram of shape $\beta$ with exactly $\{1,\dots, n\}$ is a \emph{dual immaculate  recording tableau} (DIRT)  if it has the following properties:
\begin{enumerate}
\item The rows increase from left to right.
\item The row strips start in the first (leftmost) column.
\item The first (leftmost) column increases from top to bottom.

\item The  recording triple rule is satisfied.
\end{enumerate}
\end{definition}

Combining Propositions~\ref{prop:RowStrips}~and~\ref{recTabTripRuleProp}, we get the following corollary.

\begin{cor}{\label{cor:dirt}}
If $U$ is a standard immaculate tableau and $Q$ is the recording filling obtained by insertion of the immaculate reading word $rw_{\dI}(U)$, then $Q$ is a DIRT.
\end{cor}

Let $U$ be a standard immaculate tableau and let $P$ be the standard Young composition tableau obtained from insertion of the immaculate reading word $rw_{\dI}(U)$.
Recall that the \emph{immaculate descent set, $Des_{\dI}(U)$}, of a standard immaculate tableau $U$ is the subset of
 $\{1, \hdots, n-1\}$ consisting of all entries $i$ of $U$ such that $i+1$ appears strictly above $i$ in $U$; the \emph{Young descent set, $Des_{\YQS}(P)$}, of a standard Young composition tableau $P$ is the subset of $\{1, \hdots, n-1\}$ consisting of all entries $i$ of $P$ such that $i+1$ appears weakly to the left of $i$ in $P$.  We will now show that insertion preserves the descent set.  We  begin with two lemmas.

\begin{lemma}\label{desPreservedLem1}
Let $U$ be a standard immaculate tableau and let $P$ be the tableau obtained by insertion from the immaculate reading word $rw_{\dI}(U)$.  Suppose that $i\in Des_{\dI}(U)$.  Then we have the following.
\begin{enumerate}
\item[(a)] When $i$ is initially inserted, it is weakly to the right of $i+1$.
\item[(b)] If $i$ is inserted into the same column as $i+1$ and is below $i+1$, then it is in the first (leftmost) column.
\item [(c)]  If $i$ is bumped during the insertion process, then it is still weakly right of $i+1$.
\item[(d)]   We have $i\in Des_{\YQS}(P)$.
\end{enumerate}
\end{lemma}

\begin{proof}
(a) Since  $i\in Des_{\dI}(U)$, we know that $i+1$ appears in a row above $i$ in $U$.  It follows that $i+1$ is inserted before $i$.   If the insertion of $i$ terminates before it
reaches the cell containing $i+1$, then $i$ must be weakly right of $i+1$.  If $i$ bumps $i+1$, then $i$ is still weakly right of $i+1$ as bumping moves elements weakly left.  Finally, if $i$ reaches $i+1$ and cannot bump $i+1$ then either the element to the left of $i+1$ is larger than $i$ or $i+1$ is in a row by itself.  Since rows increase, it cannot be that the element to left of $i+1$ is larger than $i$ since then it would also be larger than $i+1$.  Thus, $i+1$ must be in a row by itself.  It follows that $i$ must be put into the first (leftmost) column and so is weakly right of $i+1$.

(b) From the proof of (a), the only way for this to happen is if $i$ does not bump $i+1$.  However, in that case both $i$ and $i+1$ are in the first column as desired.

(c)  First, note that if $i+1$ is ever bumped  then it moves weakly left.  Therefore by (a) and (b)  before $i$ is ever bumped it must either precede $i+1$ in reading order or it is in the first (leftmost) column with $i+1$.  Since elements in the first (leftmost)  column cannot be bumped, if $i$ is ever bumped it must precede $i+1$ in reading order.  Applying the same reasoning as in (a) shows $i$ is weakly right of $i+1$.

(d)  Note that if $i+1$ is ever bumped and is weakly left of $i$, then it remains weakly left since bumping always moves weakly left.  From (c), we know that bumping $i$ always keeps $i$ weakly right of $i+1$.  Combining this with (a), we see that $i$ is always weakly right of $i+1$ and so $i\in Des_{\YQS}(P)$.
\end{proof}

\begin{lemma}\label{desPreservedLem2}
Let $U$ be an immaculate tableau and let $P$ be the tableau obtained by insertion from the immaculate reading
word $rw_{\dI}(U)$.  Suppose that $i\not\in Des_{\dI}(U)$.  Then we have the following.
\begin{enumerate}
\item[(a)] When $i+1$ is inserted it is strictly to the right of $i$.
\item[(b)] If $i+1$ is bumped, it is still strictly to the right of $i$.
\item[(c)] We have $i\not \in Des_{\YQS}(P)$.
\end{enumerate}
\end{lemma}

\begin{proof}
(a) Since $i$ is not in the descent set of $U$,  $i+1$ is inserted after $i$.  Note that there must be something to the right of $i$ which is larger than $i+1$; recall that we consider empty spaces next to filled cells as containing $\infty$.  If $i+1$ passes through
the cell to the right of $i$, it will bump its entry or the insertion will terminate with the placement of $i+1$ in this position.  If $i+1$ is inserted before the process reaches this cell, then $i+1$ is still strictly to the right of $i$.

(b)  First note that if $i$ is immediately to the left of $i+1$, then $i+1$ cannot be bumped.  This is because if $a$ bumps $i+1$, then $a<i+1$.  Since $i$ and $i+1$ are adjacent it must be that $i<a$.  It follows that $i<a<i+1$, which is impossible.

Moreover, note that if $i+1$ is in the column directly to the right of the column containing $i$ in $P$, then $i+1$ must appear in a row weakly above the row containing $i$.  This is because if it did not, the tableau would not satisfy the triple rule as $i+1>i$ and whatever is immediately to the right of $i$ is larger than $i+1$.

If $i+1$ is bumped, then it will be inserted into the tableau obtained by removing everything above and to the right.  Since $i$ is strictly to the left of $i+1$ and is not immediately to the left, it must be that we are now inserting $i+1$ into a tableau whose Young reading word contains the entry in the position immediately to the right of $i$.  Moreover, this position contains an element larger than $i$ and so also larger than $i+1$.  Thus we may apply the same reasoning as in the proof of  (a).

(c) When the $i+1$ is inserted it is strictly to the right of $i$ by part (a).  If $i+1$ gets bumped, it will still be strictly to the right of $i$ by  (b).  Finally, if $i$ gets bumped, it can only move weakly left and so will still be strictly left of $i+1$.  Thus, $i\not\in Des_{\YQS}(P)$.
\end{proof}

Combining Lemma~\ref{desPreservedLem1} and Lemma~\ref{desPreservedLem2} we get the following.

\begin{proposition}\label{desPreservedProp}
Let $U$ be a standard immaculate tableau and let $P$ be the Young composition tableau obtained from insertion of the reading word $rw_{\dI}(U)$.  Then $Des_{\dI}(U) = Des_{\YQS}(P)$.
\end{proposition}

{Proposition} \ref{desPreservedProp} will be critical in the proof of Theorem  \ref{thm:main} because it implies that the fundamental quasisymmetric functions associated to the standard immaculate tableaux and the fundamental quasisymmetric functions associated to the standard Young composition tableaux obtained from insertion are the same.

We now define an algorithm which we call \emph{rapture} to remove an entry from the end of a row of a Young composition tableau $T$.  We will use terminology such as ``rapture, virtuous, eviction, escape route, and DIRT'' in a tongue-in-cheek manner as a play on the terminology and origins of the ``\emph{immaculate} quasisymmetric functions'' and the quasisymmetric Schur functions (which were originally generated by ``skyline fillings'' with ``basements'').  We begin with a definition.

\begin{definition}
An entry $T(i,j)$ in a Young composition tableau $T$ is said to be \emph{virtuous} if the following hold.
\begin{enumerate}
\item The entry $T(i,j)$  is greater than all entries below it in its column.
\item The entry $T(i,j)$ is at the end of its row.
\item Any other element which is at the end of its row in column $i$ is in a row above $j$.

\end{enumerate}
\end{definition}

  The output $\mTp$ of the following procedure $k \leftarrow T$ is a Young composition tableau $T'$
together with either a positive integer $m$
such that $T=m \rightarrow T'$ or $m=\infty$.
If $m=\infty$, then there is no way to insert a positive integer into $T'$ to produce $T$.

\begin{procedure}{\label{proc:rapture}}
Let $T$ be a Young composition tableau and let
$(c_1,d_1),(c_2,d_2),$ $\dots, (c_\geed, d_\geed)$ be the corresponding cells of the augmented tableau $\bar{T}$ listed in the Young reading order and let
$k=\bar{T}(c_j,d_j)$  where $k$ is a virtuous entry of some row of $T$.
If $c_j\not=1$, let $\bar{T}(c_j,d_j)=\infty$, delete the $\infty$ immediately to the right of $(c_j,d_j)$, and set $k_0:=k$.  If $c_j=1$ then $k$ is in a row by itself.  In that situation, shift all the rows above row $d_j$ down by one to remove the gap created by the removal of $k$.  In either case, this removes $k$ from the diagram and initiates the procedure.  Let $i$ be the largest integer less than $j$ such that $\bar{T}(c_i-1,d_i) \le k_0 \le \bar{T}(c_i+1, d_{i})$, where if $(c_i+1,d_i)$ is not in the augmented diagram we set $\bar{T}(c_i+1,d_i):=\infty$.

\noindent
{\bf Case 1.}  If $\bar{T}(c_i,d_i)=\infty$, then place $k_0$ in cell $(c_i,d_i)$, set $m=\infty$ and $T'= T$, and terminate the procedure.

\noindent
{\bf Case 2.}  If $\bar{T}(c_i,d_i) \neq \infty$ and $\bar{T}(c_i,d_i) \ge k_0$, then continue the procedure with $k$ as is, moving on to the next smaller
$i$ such that $\bar{T}(c_i-1,d_i) \le k_0 \le \bar{T}(c_i+1, d_i)$.

\noindent
{\bf Case 3.}  If $\bar{T}(c_i,d_i) \neq \infty$ and $\bar{T}(c_i,d_i)<k_0$, then set $k=\bar{T}(c_i,d_i)$, replace $\bar{T}(c_i,d_i)$ by $k_0$, and repeat the procedure with this new $k$ and $\bar{T}$.  In such a situation, we say that $\bar{T}(c_i, d_i)$ is \emph{evicted}.

\noindent
If no such $i$ exists, or all such $i$ fall under Case 2, set $m=k_0$, $T'=T\setminus {(c_j,d_j)}$, and terminate the procedure.

\end{procedure}

The initial cell containing $k$ together with the sequence of cells that contain elements which are evicted during the rapture $k \leftarrow T$ are called the {\it escape route} for the rapture procedure.  Note that the procedure could be applied to an unvirtuous entry.  However, such a procedure would produce a filling that is not a Young composition tableau.

\begin{figure}
\begin{center}
\begin{tikzpicture}
\node at (1.25,0) {$8$};
\node at (1.75,0) {$\leftarrow$};
\node at (3,0) {$ \tableau{6 & {\bf 7} & \infty \\ 3 & 4 & {\bf 5} & \infty\\ 2 & {\bf 8}& \infty }$};
\node at (5,0) {$=$};
\node at (5.75,0) {$\Bigg ($};
\node at (7,0) {$ \tableau{6 & { 8} & \infty \\ 3 & 4 & { 7} & \infty\\ 2 &  \infty }$};
\node at (8,-.25) {$,$};
\node at (8.25,0) {$5$};
\node at (8.5,0) {$\Bigg )$};
\end{tikzpicture}\caption{Rapture of 8 from an augmented Young composition tableau on the left produces the  augmented tableau on the right.  The elements which are evicted are in bold.}\label{fig:rapture}
\end{center}

\end{figure}

\begin{example}\label{raptureExample1}
As an example of rapture, consider the leftmost tableau in Figure~\ref{fig:rapture}.  In the rapture of the 8, the first element which is evicted is the $7$.  The $8$ replaces the $7$ and the rapture continues with the $7$.   Then the $7$ evicts the $5$.  Finally, $5$ cannot evict any other elements so $5$ is the output.  The reader may have noticed that this example of rapture is  the inverse of Example~\ref{insertExample}.  Moreover, the escape route  in this example is the reverse of the bumping path in that example.  It is not a coincidence that rapture is the inverse of insertion as we will see in Theorem~\ref{lem:UninsertInverse}.
\end{example}

Before we consider this theorem, let us consider one more example of rapture, where the rapture starts in the leftmost column.

\begin{figure}
\begin{center}
\begin{tikzpicture}
\node at (1.25,0) {$3$};
\node at (1.75,0) {$\leftarrow$};
\node at (3,0) {$ \tableau{4&5 & \infty \\ {\bf 3} & \infty\\ 1 & {\bf 2}& \infty }$};
\node at (5,0) {$=$};
\node at (5.75,0) {$\Bigg ($};
\node at (6.75,0) {$ \tableau{4&5 & \infty \\  1 & 3& \infty }$};
\node at (7.5,-.45) {$,$};
\node at (7.75,0) {$2$};
\node at (8.,0) {$\Bigg )$};
\end{tikzpicture}
\caption{Rapture of 3 from an augmented Young composition tableau on the left produces the  augmented tableau on the right.  The elements which are evicted are in bold.}\label{fig:rapture2}
\end{center}

\end{figure}

\begin{example}
Consider the leftmost tableau in Figure~\ref{fig:rapture2}.  When 3 is raptured, the top and bottom rows come together (i.e., we remove the empty row).  Then the 3 evicts the 2.  The rapture then continues with the 2.  Since 2 cannot evict any elements, the output is $2$.  Note that once again rapture is the inverse of insertion.  However, continuing this rapture procedure would produce a word which could not have come from a standard immaculate tableau.  This is because rapturing another element from the augmented tableau on the right of Figure~\ref{fig:rapture2} would have to output a number larger than 2.   But then this word cannot be a reading word of an immaculate tableau  since the leftmost column of the corresponding immaculate tableau would not be decreasing.
\end{example}

\begin{lemma}
Let $k$ be a virtuous element of the Young composition tableau $T$.  If  $k\leftarrow T = (T',m)$ then $T'$ is a Young composition tableau.
\end{lemma}
\begin{proof}
Suppose on the contrary that $T'$ is not a Young composition tableau.  It is clear that rapturing any element at the end of a row does not change that the rows increase and that the leftmost column is decreasing from top to bottom.  Thus, it must be the case that $T'$ is not a Young composition tableau because it does not satisfy the Young composition triple rule.  Since $T$ is a Young composition tableau in every subarray as below with $a\geq b$ we have $a>c$.

\begin{center}
$\tableau{b&c}$

\vspace{12 pt}

$\hspace{12 pt} \tableau{a}$
\end{center}

Since $T'$ does not satisfy the Young composition triple rule one of the previous subarrays must have changed so that in $T'$, $a\geq b$ and $a\leq c$.  There are two possible ways for this to happen.   One possibility is that in $T$, $a<b$, but in $T'$, $a$ is evicted by some element, say $d$, and $d\geq b$, but  $d\leq c$.  We note that in this case only $a$ can be evicted.  This is because if $b$ is also evicted the triple rule would trivially be satisfied, since $b$ precedes  $a$ in reverse reading order and the elements in the escape route strictly decrease.  The same reasoning implies that $c$ cannot be evicted either.  The other possibility is that in $T$, $a\geq b$ and $a>c$, but $c$ is evicted by an element say $d$ and $a<d$ (note $c$ cannot start the rapture as $c$ is not virtuous in $T$).  Again, in this situation $a$ cannot be evicted as then the triple rule would be satisfied in $T'$ and $b$ cannot be evicted because only one entry from each row can be evicted during a given rapture.  Note that these are the only two possibilities; if $b$ is evicted $b$ is replaced by something larger, say $d$.  Therefore if the triple rule was satisfied in $T$, and $b$ is evicted, $a',d,$ and $c'$ still satisfy the triple rule, where $a'$ and $c'$ are the entries in the cells formerly occupied by $a$ and $c$ respectively. We now consider the two cases.

{\bf Case 1.} Suppose that the element $a$ is evicted by $d$ and $b$ and $c$ are not evicted.   Moreover,  suppose $d\geq b$, but $d\leq c$.  Since $d$ evicts $a$, we have $d>a\geq b$.  Thus $d>b$ and so if $d$ passed by the cell containing $b$ it would evict it.  Since $b$ is never evicted, it must be that $d$ is either in the column containing $b$ and above $b$ or in the column containing $a$ and below $a$.   First, note that it cannot be in the column containing $a$ as $d>b$, but $d<c$ which would violate the triple rule.  Thus, it must be in the column containing $b$ and above $b$.  We now have to break into two cases depending on if $c=\infty$ or if $c<\infty$.

{\bf Case 1a.} Suppose that $c=\infty$.  If $d$ is at the end of the row it is in, then $d$ is not virtuous and so the rapture cannot start with the $d$. In fact, nothing in this column above $b$ is virtuous.  It follows that something must evict $d$.  Any element which evicts $d$ would also evict $b$ which is impossible.  This implies that  the rapture must have started between $b$ and $d$.  However, no element above $b$ is virtuous and so the rapture cannot start there.

{\bf Case 1b.} Suppose that $c<\infty$.  In this case, $d$ is not at the end of its row since if it was the triple rule would be violated with the $c$ and $d$.  It follows that $d$ is not virtious and so $d$ must be evicted.  Let $e$ be the element immediately to the right of $d$.  Then $c\geq d$ and $c>e$.  Suppose $f$ is the element which evicts $d$.  Then $d<f\leq e$.  It follows that $b<f<c$.  Thus, $f$ cannot pass by $b$ as it would evict it.   Therefore $f$ is between $b$ and $d$ in reverse reading order. Since $c>f$, it cannot be that $f$ is the last element in its row as the triple rule would be violated by $c$ and $f$.  Thus, $f$ is not virtuous and must be evicted.  Apply the same reasoning to see that the rapture must start between $b$ and $d$ and reading order.  However, this cannot happen as the element which starts the rapture would need to be smaller than $c$ and so could not be at the end of its row.

We conclude that Case 1 cannot  occur.  We now consider Case 2.

{\bf Case 2.}  Suppose that $c$ is evicted by $d$ and $a$ and $b$ are not evicted. Moreover, suppose that $a\geq b$, but that $a\leq d$.   We break into  two cases depending on if $d=a$ or if $d>a$.

{\bf Case 2a.} Suppose that $d=a$ and this $a$ is in the  $j^{th}$ column and $k^{th}$ row of $T$.  We now break into subcases depending on if there is an element in the $k^{th}$ row of $T$ which is strictly less than $a$ or if no such element exists.

{\bf Subcase 2a(I).}  Suppose that there is an element in the $k^{th}$ row of $T$ which is strictly less than $a$.  Let $(i,k)$ have the property that $i$ is the smallest number such that $T(i,k)=a$ and $T(i-1,k)=a_0<a$.  In a Young composition tableau, the triple rule forces all the elements in a column to be distinct.    Moreover,  if an element appears multiple times in a row, then only the rightmost instance of this element  can be evicted.  Thus, if $d=a$ is to evict $c$, it must be the case that $d$ is in a column to the left of the $i^{th}$ column.  Since $d=a>a_0$, if $d$ passed by the cell $(i-1, k)$ it would evict $a_0$.   As this does not happen, $d$ must be in a position above $a_0$ in column $i-1$.    Thus, we have a situation depicted in the subarray  that follows.

\begin{center}
$\tableau{d} \hspace{60 pt}$

\vspace{12 pt}

$ \tableau{a_0& a & a &a & \cdot \cdot &a }$
\end{center}

Since $a\geq d$, the triple rule forces the cell immediately right of $d$ to be nonempty.  Let this entry to the right of $d$ be $d'$.  The triple rule implies that $a>d'$.   However, $d\leq d'$ since rows increase and so $a=d\leq d'<a$ which is impossible.

{\bf Subcase 2a(II).}  Suppose that every element in the $k^{th}$ row of $T$ is greater than or equal to $a$. Since elements in the  columns of a Young composition tableau are distinct, this means that one of the $a$'s in this row must be evicted. However, as mentioned earlier this would have to be the rightmost such $a$.  This would force the $a$  in the column with $c$ to be evicted, but this cannot happen.

{\bf Case 2b.} Now suppose that $d>a$.  We break into two subcases depending on if $a$ is at the end its row or not.

{\bf Subcase 2b(I).} Suppose that $a$ is at the end of its row in $T$.  Then since $d>a$, if $d$ scanned $a$ then $d$ would evict $a$.  Thus, $d$ must appear between $a$ and $c$ in reverse reading order.  Since $a$ is at the end of its row, $d$ is not virtuous and so must be evicted.  Any element which evicts $d$ would be larger than $a$ and so must be between $a$ and $c$ in reverse reading order for otherwise this element would evict $a$.  Continue this reasoning to see that the rapture must start between $a$ and $d$.  However, this is impossible since no element above $a$ in this column is virtuous.

{\bf Subcase 2b(II).} Suppose that $a$ is not at the end of its row.  Let $x$ be the element immediately right of $a$ in $T$.  Then since $x\geq a>c$, the triple rule implies that  $c$ cannot be at the end of its row.  Let $y$ be the element which is immediately right of $c$ in $T$.  By the triple rule, $x>y$.  Now $d$ is to evict $c$ and so $a<d \le y<x$.  It follows that $a<d<x$ and so if $d$ passed by $a$ it would evict $a$.  Since this does not happen, $d$ must appear between $a$ and $c$ in reverse reading order.  The fact that $x>d$ implies that $d$ is not at the end of its row as this would force a triple rule violation.  Let $e$ be the element immediately right of $d$.  Note that by the triple rule, $x>e$.  Since $d$ is not virtuous, it must be evicted.  If $f$ evicts $d$, then $a<d<f \le e<x$.  Thus, $f$ would evict $a$ if it passed by it.  Therefore $f$ is between $a$ and $c$ in reverse reading order.  Apply the same reasoning as before to see that $f$ must be evicted.  Keep applying this reasoning to see that the rapture must start between $a$ and $d$, but that this element which would start the rapture is not virtuous, a contradiction.
\end{proof}

\begin{theorem}{\label{lem:UninsertInverse}}
When the rapture of $k$ from a Young composition tableau $T$ produces the output $\mTp$, where $m$ is an integer $m < \infty$, the insertion of $m$ into $T'$ produces the Young composition tableau $T$.  Similarly, if an integer $r$ is inserted into a Young composition tableau $S$ to produce a Young composition tableau $S'$, then rapturing the entry in the new cell in $S'$ produces the pair $(S,r)$.  That is, rapture is the inverse of insertion.
\end{theorem}

\begin{proof}
We must prove that rapture reverses insertion, and that insertion reverses rapture.  That is, we first show that if an integer $m$ inserted into a Young composition tableau $T'$ produces a Young composition tableau $T$, then rapturing the entry in the newly created cell in $T$ results in the output $\mTp$.   Then we show that if $\mTp$ is the output when an entry $k$ is raptured from $T$, then $T$ is the Young composition tableau obtained by inserting $m$ into $T'$.

First consider the insertion
$m \rightarrow T'$ with bumping path $(c_1,d_1), \hdots, (c_\geef,d_\geef).$
We show the entry $k=T(c_\geef,d_\geef)$ is virtuous and the rapture $k \leftarrow T$ produces $\mTp$ with
escape route $(c_\geef,d_\geef), \hdots, (c_1,d_1).$
Suppose $k$ fails to satisfy the first virtuous condition.  Then $k$ cannot be in the leftmost column since the leftmost column increases bottom to top.
If there exists an entry $i=T(c_\geef,h)$ with $h<d_\geef$
such that $i > k$, then $i$ was not impacted by the insertion $m \rightarrow T'$ since the location of $i$ is after $(c_g,d_g)$ in reading order.  Now, $T'(c_\geef-1,d_\geef) \le T(c_\geef,d_\geef)=k<i=T'(c_\geef,h)$
and $\bar{T'}(c_\geef,d_\geef)=\infty.$
But then the entries in cells $(c_\geef-1,d_\geef), (c_\geef,h)$, and $(c_\geef,d_\geef)$ force $T'$ to violate the triple rule.  Therefore $k$ satisfies the first virtuous condition.

The entry $k$ at which insertion terminates must be at the end of its row, so the second virtuous condition is satisfied by construction.  To see that the third virtuous condition is satisfied, we must check that there are no rows below the row containing $k$ of the same length as the row containing $k$.  If there were such a row, say row $h$, then the rightmost entry $e=T'(c_g,h)$ in this row must be less than $k$ by the first virtuous condition (which is satisfied by the above argument).  Since the insertion did not terminate at the cell $(c_g+1,h)$, the entry which scanned cell $(c_g,h)$ during the insertion of $m$ must have been strictly less than $e$, and hence strictly less than $k$.  Therefore $k$ must have been bumped from a position after $(c_g+1,h)$ in reading order.  But if $k$ were bumped from a lower row, say row $i$, of column $c_g+1$, then $T'(c_g+1,h)=k, T'(c_g,h)=e,$ and $\bar{T}'(c_g+1,h)=\infty$ would violate the triple condition in $T'$ since $e<k$ and $k < \infty$.  Therefore $k$ must have been bumped from column $c_g$ but a higher row, row $r$, than $d_g$ of $T'$.  Since the entry $a=T'(c_g-1,r)$ immediately to the left of $k$ in $T'$, together with $k$ and $T'(c_g,h)$ satisfy the triple rule in $T'$, we have $a \ge e$ and by the above argument the entry bumping $k$ must have been bumped from column $c_g$.  (Note that entries in the leftmost column cannot be bumped or evicted and therefore the entry $T'(c_g-1,r)$ does in fact exist.)  Repeating this argument shows that each entry bumped from column $c_g$ must have been larger than $e$, but we also saw that the entry scanning $T'(c_g+1,h)$ was less than $e$, and every entry bumped from a lower row of column $c_g+1$ must be less than $e$.  These three statements together are not possible since this implies that the entry which moves from column $c_g+1$ to column $c_g$ must be both less than $e$ and greater than $e$ at the same time, a contradiction.  Therefore $k$ must be virtuous.

Let $(c_1',d_1'), (c_2',d_2'), \hdots , (c_h', d_h')$ be the escape route for the rapture of $k$.  If this escape route is identical to the bumping path for the insertion of $m$ (that is, $(c_1',d_1')=(c_{g},d_{g}), (c_2',d_2')=(c_{g-1},d_{g-1}), \dots ,(c_h',d_h')=(c_1,d_1)$) then the result of the rapture of $k$ is the pair $(T',m)$, as desired.  Therefore we may assume that escape route and the bumping path differ at some cell.  Consider the first cell in reverse reading order (starting from $(c_g,d_g)$) at which the escape route and the bumping path differ.  Either this cell is in the escape route and not the bumping path or this cell is in the bumping path and not the escape route.

Assume first that $(c_i',d_i')$ is the first cell in reverse reading order in the escape route but not the bumping path.  This means that during the insertion of $m$ into $T'$, the entry scanning this cell did not bump it, but went on to bump the next entry on the escape route.  We know that the entry $e$ contained in $T(c_{i-1}',d_{i-1}')$ after insertion is the entry scanning $T(c_i',d_i')$ during the rapture of $k$, since the escape route agrees with the bumping path up to this point (in reverse reading order).  Furthermore, the bumping path does not include any cells between $(c_i',d_i')$ and $(c_{i-1}',d_{i-1}')$ so in fact $e$ is the entry which scanned the cell $(c_i',d_i')$ during insertion.  Since $e$ evicts $T(c_i',d_i')$ during rapture,  we have $T(c_i'-1,d_i') \le  e$ and $e > T(c_i',d_i')$.  Since the entry in $T(c_i'-1,d_i')$ after eviction is the same as the entry in $T'(c_i'-1,d_i')$ (by equality of bumping path and escape route up to $(c_i',d_i')$), we have $e \ge T'(c_i'-1,d_i')$.  Since $e$ does not bump $T'(c_i',d_i')$, we must have $e  \ge T'(c_i',d_i')$, a contradiction.  Therefore if the escape route and the bumping path differ, their first difference (in reverse reading order) involves a cell which is in the bumping path but not the escape route.

Assume next that $(c_i,d_i)$ is the first cell in reverse reading order in the bumping path but not in the escape route.  Then $e=T'(c_i,d_i)$ is the entry bumped from this cell and hence $T'(c_i-1,d_i) \le e \le T'(c_i+1,d_i)$ and the entry in $(c_i,d_i)$ during rapture is smaller than $e$.  Since the bumping path and escape route agree after that point, $e$ is the entry scanning cell $(c_i,d_i)$ during rapture.  But then $T'(c_i-1,d_i) \le e \le T'(c_i+1,d_i)$ and $e$ is larger than the entry in cell $(c_i,d_i)$ during rapture.  So $e$ should evict this entry and we have a contradiction.

We have shown that there cannot be a first position at which the bumping path for $m \rightarrow T'$ differs from the escape route for $k \leftarrow T$, and hence the bumping path and escape route must be identical.  Therefore rapturing $k$ from $T$ produces the pair $(T',m)$ as desired.

To go the other way, assume that $k$ is raptured from a Young composition tableau $T$ to produce $(T',m)$, and then $m$ is inserted back into $T'$.  We will prove that the result is $T$ by showing the escape route for the rapture of $k$ is identical to the bumping path for the insertion of $m$.  This proof is similar to the above argument so we use broader strokes here to provide intuition supplementing the precision in the argument above.

If the escape route and bumping path are identical then the proof is complete.  Therefore we may assume that they differ at some cell.  Either this cell is in the escape route and not the bumping path or this cell is in the bumping path and not the escape route.

Assume first that $(c_i',d_i')$ is the first cell in reading order that is in the escape route but not the bumping path.  The next entry it evicts is in the bumping path, as are all the remaining entries evicted as the rapture comes to a conclusion.  Therefore the entry $e$ to scan $(c_i',d_i')$ of $T'$ during insertion is the entry that was evicted from $(c_i',d_i')$ of $T$ during rapture.  This entry $e$ scanning cell $(c_i',d_i')$ during insertion must be smaller than $T'(c_i',d_i')$ and was situated in this cell before rapture; hence it satisfies all conditions necessary to bump this entry, a contradiction.

Assume next that $(c_i,d_i)$ is the first cell in reading order in the bumping path but not the escape route.  Since the bumping path and escape route agree up to this point, this means that we are bumping an entry which was not evicted, hence the entry $e$ doing this bumping passed by this cell without evicting $T(c_i,d_i)$ during rapture.  But since $e$ bumps the entry in $T'(c_i,d_i)$, we must have $e$ greater than or equal to $T(c_i-1,d_i)$.  But then $e$ would have evicted the entry in cell $(c_i,d_i)$ of $T$ since $e$ is greater than this entry, a contradiction.

We have shown that there cannot be a first position at which the bumping path for $m \rightarrow T'$ differs from the escape route for $k \leftarrow T$, and hence the bumping path and escape route must be identical.  Therefore inserting $m$ into $T'$ produces $T$, as desired.

We have shown that when we insert and then rapture we get the same bumping and escape route sequence.  We have also shown that if we rapture and then insert we get the same bumping path and escape route.  It follows that insertion and rapture are inverses of each other.
\end{proof}

We now define a procedure that uses rapture to determine the immaculate tableau associated with a pair $(T,Q)$ where $T$ is a Young composition tableau and $Q$ is a DIRT of the same shape.

\begin{procedure}{\label{proc:uninsert}}
Given a pair $(T,Q)$ with $T$ a Young composition tableau and $Q$ a DIRT both of shape $\beta$.  Set $w:=\emptyset$, the empty word.
\begin{enumerate}
\item Suppose that the largest entry in $Q$ occurs at position $(i,j)$.  Set $z=T(i,j)$, erase the entry in $(i,j)$ from $Q$ to obtain the DIRT $Q'$, and rapture $z$ from $T$ to obtain the pair $\mTp$.
\item Prepend $m$ to $w$, replace $T$ by $T'$, replace $Q$ by $Q'$, and repeat step 1.
\item Once all entries have been removed from $Q$ (and thus from $T$), the procedure is complete and the output is the word $w$.
\end{enumerate}
\end{procedure}

In the following lemma, we will prove that if  $Q$ is  a DIRT, then rapture will always produce a word.  We must therefore prove that is a well-defined procedure; i.e., that all entries in $T$ identified as candidates for rapture are in fact virtuous.  We also prove that the word contains no infinite letters.   Moreover, the resulting word has several additional important properties.  In particular, we have the following.

\begin{lemma}{\label{lem:UninsertWord}}
Suppose $T$ is a Young composition tableau and $Q$ is a DIRT of the same shape as $T$ with row strip shape $\alpha$.  If we apply Procedure \ref{proc:uninsert} to $(T,Q)$ we get the reading word of an immaculate tableau of shape $\alpha^{rev}$.
\end{lemma}

\begin{proof}
We first prove that Procedure \ref{proc:uninsert} is well-defined.  To see that the first entry $z=T(i,j)$ identified for rapture from $T$ is virtuous, we must first prove that $z$ is greater than all entries $T(i,\geeg)$ such that $\geeg<j$.  Since $Q(i,j)$ is the largest entry in $Q$, we have $Q(i,\geeg) < Q(i,j)$ for all $\geeg<j$ such that $Q(i,\geeg)$ is non-empty.  For every such non-empty $Q(i,\geeg)$, the recording triple rule implies that since $Q(i,j) > Q(i,\geeg)$, we must have $Q(i,j)>Q(i+1,\geeg)$.  In particular, $Q(i+1,\geeg)$ is non-empty since empty cells are considered to contain the value infinity.  This means that $T(i+1,\geeg)$ is also non-empty.  The Young composition triple rule implies that $T(i+1,\geeg) < T(i,j)$, since $T(i+1,j)=\infty$ as this cell is empty.  But $T(i,\geeg) < T(i+1,\geeg)$, so $T(i,\geeg) < T(i,j)$.   Note also that $T(i,j)$ is the last entry in its row and it is the lowest such element since otherwise the the recording triple rule of $Q$ would not be satisfied. Therefore $T(i,j)$ is virtuous.

Next we show that removing the largest entry from a DIRT still produces a DIRT.  It is clear that the rows still increase, the row strips still start in the leftmost column,  and the leftmost column is increasing top to bottom.  We must check that the recording triple rule is still satisfied.  When the largest entry is removed from $Q$, that cell is now considered an infinity.  The triples not involving this cell remain the same, so it is enough to check only triples involving the removed cell.  Such triples are situated as below, with either $a=\infty$ or $c=\infty$.

\begin{center}
$\tableau{a}$

\vspace{15 pt}

$\hspace{15 pt} \tableau{b & c}$

\end{center}

If $a=\infty$, then $a>b$ and $a>c$, so the triple rule is satisfied.  If $c=\infty$, then we must show that $a<b$.  But if $a>b$, then $a$ would have been greater than the entry removed from the cell containing $c$ in $Q$.  So $c$ would not have been removed from $Q$ since it was not the largest entry in $Q$.  Therefore $a<b$.

Removing the largest entry from a DIRT produces a DIRT and rapturing a virtuous entry from $T$ produces a Young composition tableau.  Therefore each candidate for rapture during Procedure \ref{proc:uninsert} is virtuous by the argument in the first paragraph of this proof.

Next, we prove that each rapture produces an integer less than infinity.  Assume not.  Then at some point during the rapturing (Procedure~\ref{proc:rapture}) of the entry $z=T(i,j)$ from the YCT, we reach the situation described in Case 1 of Procedure~\ref{proc:rapture}; that is, $\bar{T}(c_i-1, d_i) < k_0$ and $\bar{T}(c_i,d_i) = \infty$.  Note that this means $k_0$ is the entry that evicts $\infty$.

If the rapture procedure began somewhere before the cell $(c_1-1, d_i)$ in reverse reading order (meaning the rapturing procedure scanned the entry in cell $(c_i-1,d_i)$), then it would evict $\bar{T}(c_i-1,d_i)$ as it must be greater than or equal to $k_0$ since the escape route forms a decreasing sequence.  But $T(c_i-1,d_i) < k_0$, so the entry in cell $(c_i-1,d_i)$ was not impacted by this rapture.  Therefore the rapture began in some cell after the cell $(c_i-1,d_i)$.  If the rapture began in column $c_i-1$ in a row $r$ above row $d_i$, then  the element which starts the rapture is not virtuous. 
This means that the rapture must have started in column $c_i$ in a row $s$ lower than row $d_i$.  However, in that case the entries $T(c_i-1,d_i), T(c_i,s),$ and the infinity immediately to the right of $(c_i-1,d_i)$ violate the Young composition tableau triple rule.  Therefore this cannot happen and thus Procedure \ref{proc:uninsert} always outputs a number less than infinity.

Finally, suppose that $\dee_1$ and $\dee_2$ are consecutive elements ($a_2=a_1+1$) of some row strip in $Q$.  Specifically, let
$\dee_2=Q(j,\geeh)$ and $\dee_1=   Q(h,\geeg)$  where 
$h<j$.

Moreover, suppose that when we rapture $f_2=T(j,\geeh)$ and $f_1=T(h,\geeg)$ (the entries in $T$ corresponding to $\dee_2$ and $\dee_1$, respectively),
then the outputs are $e_2$ and $e_1$, respectively.  We will show that $e_1<e_2$.

Suppose that the escape route for $f_2$ is $(c_1,d_1),(c_{2},d_{2}),\hdots, (c_\geeh,d_\geeh),$
and set $u_\geeg=T(c_\geeg,d_\geeg)$ for $1 \le \geeg \le \geeh.$
Let $\cee_\geeg=T(c_\geeg-1,d_\geeg)$ be the element to the left of $u_\geeg$ just before $u_\geeg$ was evicted while rapturing $f_2$.
When we rapture $f_1$, the recording triple rule implies that we pass by the positions containing the $\cee_k$'s.
Let $v_{k}$ be the element which passes by $\cee_k$ when rapturing $f_1$.

We break into two cases depending on how many elements are evicted when rapturing $f_1$.

{\bf Case 1.} There are no evictions.

In this situation, $e_1=f_1.$
During the rapture process, $f_1$ moved past $\cee_\geeh$ without evicting it;
so $e_1=f_1<\cee_\geeh<e_2$.

{\bf Case 2.} There is at least one eviction.

We prove by induction that for $\geeg=1,\dots, \geeh-1$, we have $v_{\geeg+1} < u_\geeg$. Consider what happens when $v_{1}$ passes by $\cee_1$.
If $v_{1}$ evicts $\cee_1$, then $v_{2} \le \cee_1< u_1$.
Since the position to the right of $\cee_1$ is empty after the rapture of $f_2$, if $v_{1}$ does not evict $\cee_1$  then it must be that $v_{1}<\cee_1<u_1$.  Since evicting decreases the elements, $v_{2}\leq v_{1}<\cee_1<u_1$.  In either case, $v_2 < v_1$.

Now suppose that $v_{\geeg}<u_{\geeg-1}$. Consider what happens as $v_{\geeg}$ passes by $\cee_\geeg$ which is immediately left of $u_{\geeg-1}$ after the rapture of $f_2$.  If $v_k$ evicts $p_k$, then $v_{\geeg+1} \le  \cee_\geeg<u_{\geeg}$. If $v_k$ does not evict $p_k$, then it must be the case that $v_{\geeg}<\cee_\geeg<u_\geeg$.  This is because $\cee_\geeg$ is immediately to the left of $u_{\geeg-1}$ and $v_{\geeg}<u_{\geeg-1}$.  In either case, $v_{\geeg+1}<u_\geeg$ as claimed.

One of two things happened when we passed by $\cee_\geeh$ while rapturing $f_1$: either $\cee_\geeh$ was evicted or it was not.  If $p_g$ was evicted, then $\cee_\geeh$ appears in the sequence of evictions for the rapture of $f_1$.  Since $\cee_\geeh<u_\geeh$ and the entries involved in eviction decrease, this implies that $e_1<u_\geeh=e_2$.  If $p_g$ was not evicted, it must have been the case that $v_{\geeh}<\cee_{\geeh}$ which again implies that $v_{\geeh}<u_\geeh$ and so $e_1<e_2$.

 It is not hard to see that when we removed elements from the leftmost column of $T$ which are in the last row, these elements just come out (i.e., there is no eviction).  This is because they are the smallest elements of $T$.  Combining this with the fact that removing elements in a row strip creates a decreasing sequence of elements implies the word we get out of rapture is indeed an immaculate word of shape $\alpha^{rev}$.
\end{proof}

\section{Proof of Main Theorem and Related Results}\label{mainThmSec}

\begin{figure}
\centering
\begin{tabular}{ |c | c|  c| c| }
\hline
Standard Immaculate Tableau & SYCT & DIRT & Fundamental\\
\hline
\tableau{3 & 4\\ 1 & 2} & \tableau{3 & 4\\  1 & 2} & \tableau{1 &2 \\ 3 & 4}& $F_{(2,2)}$\\[9 pt]
\hline
\tableau{2&4\\ 1&3 }  & \tableau{2&3 \\1&4} & \tableau{1 &2 \\ 3&4}& $F_{(1,2,1)}$ \\[9 pt]
\hline
\tableau{2&3\\ 1&4} & \tableau{2&3&4 \\ 1} & \tableau{1 &2&4 \\ 3}&$F_{(1,3)}$\\[9 pt]
\hline
\end{tabular}
\caption{The three standard immaculate tableau of shape $(2,2)$ with the standard Young composition tableaux, and dual immaculate recording tableaux obtained from insertion and the fundamental quasisymmetric function associated with both the standard immaculate tableau and standard Young composition tableau. \label{insExTab}}
\end{figure}

Before we prove Theorem  \ref{thm:main}, we compute a small example that shows how the insertion algorithm gives the corresponding decomposition.  We decompose $\dI_{(2,2)}$.  Figure~\ref{insExTab} contains the three standard immaculate tableaux of shape $(2,2)$, the three standard Young composition tableaux obtained from insertion, their respective dual immaculate recording tableaux, and the associated fundamental quasisymmetric functions.

The figure shows that $\dI_{(2,2)} = F_{(2,2)}+F_{(1,2,1)}+F_{(1,3)}$.  Moreover, every standard Young composition tableau of shape $(2,2)$ and $(1,3)$ appears exactly once.  Since the the insertion map preserves the descent sets, this implies that $\YQS_{(2,2)}+\YQS_{(1,3)} =  F_{(2,2)}+F_{(1,2,1)}+F_{(1,3)}$.  Therefore, $\dI_{(2,2)}=\YQS_{(2,2)}+\YQS_{(1,3)}$.  Note also that each recording tableau we obtain is a DIRT with row strip shape $(2,2)^{rev}$.  Moreover,  each DIRT of shape $\beta$ with row strip shape $(2,2)^{rev}$ appears with each standard Young composition tableau of shape $\beta$ exactly once.  Thus we just need to count the number of DIRTs of shape $\beta$ to get the coefficient of $\YQS_{\beta}$ in $\dI_{(2,2)}$.  There is one DIRT of shape $(2,2)$ and one of shape $(1,3)$, both  with row strip shape $(2,2)^{rev}$.  Therefore, we see again that $\dI_{(2,2)}= \YQS_{(2,2)} + \YQS_{(1,3)}$.

\subsection{Proof of the main theorem and results about DIRTs}

Recall that Theorem~\ref{thm:main} states that
\begin{align}{\label{eq:MainThm}}
\dI_{\alpha} = \sum_{\beta} c_{\alpha,\beta} \YQS_{\beta},
\end{align}
where $c_{\alpha,\beta}$ is the number of DIRTs of shape $\beta$ with row strip shape $\alpha^{rev}$.

\begin{proof}[Proof of Theorem 1.1]
For $\alpha\vDash n$,  let $I(\alpha)$ be the set of standard immaculate tableaux of shape $\alpha$.  Additionally, let $Y(\alpha)$ be the set of pairs $(P,Q)$ such that $P$ is a SYCT, $Q$ is a DIRT with row strip shape $\alpha^{rev}$, and $P$ and $Q$ have the same shape.

We claim that there is a bijection, $\varphi$, from $I(\alpha)$ to $Y(\alpha)$ such that if  $\varphi(U) = (P,Q)$ then $Des_{\dI}(U) = Des_{\YQS}(P)$.  Assume for now that such a bijection $\varphi$ exists.
It follows that
\begin{equation}\label{proofEq}
\sum_{U \in I(\alpha)} F_{Des_{\dI}(U)} =\sum_{(P,Q) \in Y(\alpha)} F_{Des_{\YQS}(P)}.
\end{equation}
By Proposition~\ref{fundDecomp}, we know that the right hand side of equation~\ree{proofEq} is the right hand side of equation~\ree{eq:MainThm}.  Moreover, by Proposition~\ref{fundDecompofDI}, the left hand side of equation~\ree{proofEq} is the left hand side of equation~\ree{eq:MainThm}.   It follows that if such a bijection exists, equation~\ree{eq:MainThm} holds.

To see that our desired bijection $\varphi$ exists, begin with an arbitrary composition $\alpha$ and let $U$ be a standard  immaculate tableau of shape $\alpha$.  Recall the reading word of $U$ is given by reading the rows of $U$ from left to right, beginning at the top row and working from top to bottom.  See Figure~\ref{fig:insert} for an example.  Note that the rows of $U$ appear as the longest consecutive increasing subsequences in this reading word $rw_{\dI}(U)$, since the leftmost column entries are strictly decreasing from top to bottom.  Let $\varphi$ be the map that sends this reading word to a pair $(P,Q)$ consisting of a standard Young composition tableau $P$ and a dual immaculate recording tableau $Q$ using the insertion algorithm described in Procedure~\ref{proc:insertion}.  We know $P$ is a standard Young composition tableau because $P$ was obtained using the insertion procedure, and Corollary~\ref{cor:dirt} implies that $Q$ is a dual immaculate recording tableau.

To see that $\varphi$ is a bijection, note that its inverse is given by Procedure~\ref{proc:uninsert}.  Here, we record the resulting output from the rapture procedure to form a word.  Theorem~\ref{lem:UninsertInverse} and Lemma~\ref{lem:UninsertWord} imply that the resulting word is in fact the reading word of the unique standard immaculate tableau of shape $\alpha$ which mapped to $(P,Q)$ under insertion.  Therefore the map is a bijection, as desired.
\end{proof}

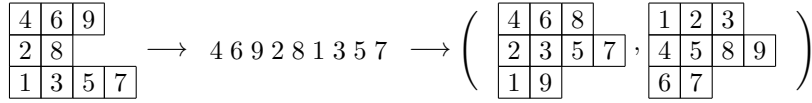
\begin{figure}
\begin{center}
\begin{tikzpicture}
\node at (0,0) {$\tableau{4 & 6 & 9 \\ 2 & 8 \\ 1 & 3 & 5 & 7}$};
\node at (1.25,0) {$\longrightarrow$};
\node at (3,0) {$4 \; 6 \; 9 \; 2 \; 8 \; 1 \; 3 \; 5 \; 7$};
\node at (4.75,0) {$\longrightarrow$};
\node at (5.25,0) {$\Bigg ($};
\node at (6.5, 0) {$\tableau{4 & 6 & 8 \\ 2 & 3 & 5 & 7 \\ 1 & 9}$};
\node at (7.5,0) {$,$};
\node at (8.5,0) {$\tableau{1 & 2 & 3 \\ 4 & 5 & 8 & 9 \\ 6 & 7}$};
\node at (9.75,0) {$\Bigg )$};
\end{tikzpicture}
\caption{The map from a standard immaculate tableau to a SYCT and the corresponding DIRT.}  \label{fig:insert}
\end{center}
\end{figure}

Let us now consider some consequences of Theorem~\ref{thm:main}.   For a composition $\alpha=(\alpha_1,\alpha_2,\dots, \alpha_\ell)$ define $j_m$ to be $\sum_{i=1}^{m-1} \alpha_{\ell-i-1}$.  Now let $\lambda$ be a partition. Let $Q$ be the filling of a diagram with shape $\lambda$ obtained from filling the $(\ell-m+1)^{th}$ row from left to right with integers in $[j_m+1, j_{m+1}]$.  It is not hard to see that $Q$ is  the unique  DIRT of shape $\lambda$ and row strip shape $\lambda^{rev}$. We call such a tableau \emph{superstandard}.

As an example if $\lambda=(3,2,1)$, then
$$
 \tableau{1\\2&3\\4&5&6}
$$
is the superstandard DIRT of shape $\lambda$ and  row strip shape $\lambda^{rev}$.

We will use the superstandard DIRTs to prove the following proposition.

\begin{proposition}\label{prop:rearrangeDIRT}
Let $\lambda$ be a partition.  If $\alpha$ is a composition such that $\alpha$ rearranges to $\lambda$, then there exists a DIRT of shape $\alpha$ and row strip shape $\lambda^{rev}$.
\end{proposition}
\begin{proof}
We first describe an action of a transposition of the form $(i, i+1)$ on a filling of a composition diagram $F$.  Let $F$ have shape $(\alpha_1,\alpha_2, \dots, \alpha_\ell)$.  If $\alpha_i <\alpha_{i+1}$, then when $(i, i+1)$ acts on $F$  we append the suffix of the $(i+1)^{th}$ row of $F$ which starts with the element in position $(\alpha_{i}+1,i+1)$ to the end of row $i$ of $F$. If $\alpha_i >\alpha_{i+1}$,  then we append the suffix of the $i^{th}$ row of $F$ which starts with the element in position $(\alpha_{i+1}+1,i)$ to the end of row $i+1$ of $F$.   If $\alpha_i = \alpha_{i+1}$ we do nothing.  Notice that the resulting filling has shape $(\alpha_1,\alpha_2,\dots, \alpha_{i-1},\alpha_{i+1},\alpha_i,\alpha_{i+2}, \dots, \alpha_\ell)$.  Moreover, it is easy to see that the elements in the columns of $F$ never change.  In fact, the order in which the elements appear in the column never changes.   We can extend this action to the full symmetric group $S_\ell$ by first writing a reduced word for the permutation and then acting one transposition at a time.

Now suppose that  $Q$ is the superstandard DIRT of shape $\lambda$ and $\pi \in S_\ell$ where $\ell$ is the number of rows in $Q$.  Let $Q'$ be the filling we obtain by applying $\pi$ to $Q$.  We claim that $Q'$ is still a DIRT with row strip shape $\lambda^{rev}$.  First note that since we are only moving elements within their column, the row strips  of $Q$ and $Q'$ are the same.  This implies that the row strip shape of $Q'$ is $\lambda^{rev}$ and condition (2) of Definition~\ref{dirtDef} is satisfied for $Q'$.  Since all the columns of $Q$ increase from top to bottom and the columns never change order when we apply $\pi$, all the columns of  $Q'$ are still increasing from top to bottom. This implies both that the leftmost column is increasing from top to bottom and that the recording triple rule is trivially satisfied for $Q'$.  This shows that conditions (3) and (4) of Definition~\ref{dirtDef} are satisfied for $Q'$.

We must now verify condition (1) of Definition~\ref{dirtDef}.  That is, we must show that the rows of $Q'$ are increasing.     Suppose that $x$ was in row $k$ in $Q$ and is in row $m$ in $Q'$.  We claim that $k\leq m$.  Suppose this was not the case.  Since the shape of $Q$ is a partition, there are $k-1$ elements below $x$ in $Q$.  Since the order of the elements in the columns of $Q'$ are the same, there are still $k-1$ elements below $x$ in $Q'$.  It follows that $k\leq m$.  In $Q$ every element weakly above and left of $x$ is less than $x$.  This fact combined with the fact that $k\leq m$ implies the element in $Q'$ directly to the left of $x$ is smaller than $x$. Thus the rows increase from left to right.  We conclude that $Q'$ is a DIRT of row strip shape $\lambda^{rev}$.

To finish the proof just note that if $\alpha$ is a rearrangement of $\lambda=(\lambda_1,\lambda_2,\dots, \lambda_\ell)$, then there is some permutation $\pi$ such that $\alpha= (\lambda_{\pi(1)},\lambda_{\pi(2)},\dots, \lambda_{\pi(\ell)})$.  Let $Q$ be the superstandard tableau of shape $\lambda$.  Apply $\pi$ to $Q$ to obtain a DIRT of shape $\alpha$ and row strip shape $\lambda^{rev}$.
\end{proof}

Note that the definition of DIRT forces certain conditions on the shape and row strip shape of a DIRT.  In particular we have the following lemma.

\begin{lemma}\label{domLem}
Let $Q$ be a DIRT of shape $\alpha$ and row strip shape $\beta^{rev}$.  Then $\beta \succeq \alpha$ in dominance order.
\end{lemma}
\begin{proof}
Consider a diagram of shape $\alpha$ that will be filled in to create a DIRT $Q$ with row strip shape $\beta^{rev}$ where $\beta=(\beta_1, \beta_2, \hdots , \beta_{\ell})$.   Fill in the rows of $\alpha$ with the row strips corresponding to $\beta^{rev}$ starting from the top so that the first row strip contains the entries $1,2, \hdots , \beta_{\ell}$, the second row strip contains the entries $\beta_{\ell}+1, \beta_{\ell}+2, \hdots, \beta_{\ell}+\beta_{\ell-1}$, etc.  Consider the step when there are $i$ more rows to be filled.  (See Figure~\ref{fillDIRTFig} for an example, where the black boxes represent cells which have already been filled.)  Then there are $\alpha_1+\alpha_2+\cdots+ \alpha_i$ positions still left to be filled which are in or below the $(\ell-i+1)^{th}$ row of $\alpha$.  Since these positions must be filled and there might also be positions above row $i$ left to be filled, it must be the case that $\alpha_1+\alpha_2+\cdots+\alpha_i \leq \beta_1+\beta_2 + \cdots + \beta_i$.   As this must hold for all $i$, we have $\beta \succeq \alpha$.
\end{proof}

\begin{figure}
\begin{center}
\begin{tikzpicture}
\node at (-.6, 1) {$\vdots$};
\node at (0,0,) {\tableau{  {\blacksquare}&{\blacksquare}&{}\\ {\blacksquare}&{\blacksquare}&{\blacksquare} &{\blacksquare} {}&{}\\  {}&{}&{} &{} {}&{}&{}&{}}};
\node at (-.6,-1.3) {$\vdots$};
\node at (-.6,-2) {\tableau{ \\{}&{}&{} &{}&{}&{}\\{}&{}}};
\end{tikzpicture}
\end{center}
\caption{The empty rows in the diagram represent rows 1 through $i$. There are also empty boxes above row $i$.  These boxes must be filled to finish constructing the DIRT.}\label{fillDIRTFig}
\end{figure}
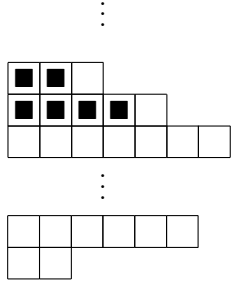

\begin{proposition}\label{partMustBeEqualProp}
Let $\lambda$ be a partition and let $\alpha$ be a composition.  There is a DIRT of shape $\lambda$ and row strip shape $\alpha^{rev}$ if and only if $\alpha=\lambda$.
\end{proposition}
\begin{proof}
($\Rightarrow$) This follows from Lemma~\ref{domLem}, the fact that the shape and row strip shape of a DIRT must have the same length, and the fact that  among all compositions of $n$ with the same length, partitions are largest in dominance order.

\noindent ($\Leftarrow$)  If $\lambda=\alpha$, then Proposition~\ref{prop:rearrangeDIRT} implies there is a DIRT of shape $\lambda$ and row strip shape $\lambda^{rev}$.
\end{proof}

\begin{cor}\label{allRearrangAppearCor}
Let $\alpha$ be a composition and let $\lambda$ be a partition.  Then $\YQS_{\beta}$ appears with positive coefficient for each rearrangement $\beta$ of $\alpha$ in the decomposition of $\dI_\alpha$  into the Young quasisymmetric Schur functions if and only if $\alpha=\lambda$. 
\end{cor}

\begin{proof}
($\Rightarrow$)  In this case, the coefficient of $\YQS_\lambda$ is positive in the decomposition of $\dI_\alpha$.  This together with Theorem~\ref{thm:main}  imply that there is a DIRT of shape $\alpha$ and row strip shape $\lambda^{rev}$.  Thus by Proposition~\ref{partMustBeEqualProp},  $\alpha=\lambda$.\\

\noindent ($\Leftarrow$) Suppose that $\alpha=\lambda$.  Then Proposition~\ref{prop:rearrangeDIRT} implies that there is a DIRT of each shape which is  a rearrangement of $\lambda$ whose row strip shape $\lambda^{rev}$.  The result then follows from Theorem~\ref{thm:main}. 
\end{proof}

\subsection{Connections to $Sym$}

The next theorem will be useful in  what follows.

\begin{theorem}[\cite{LMvW13}]\label{SchurToYQS}
Let $\lambda$ be a partition.  Then
$$
s_\lambda = \sum_{\alpha} \YQS_\alpha
$$
where the sum is over all compositions which rearrange to $\lambda$.
\end{theorem}

We note that  Corollary~\ref{allRearrangAppearCor}, Theorem~\ref{SchurToYQS}, and the fact that there is only one DIRT of shape $\lambda$ and row strip shape $\lambda^{rev}$ imply that  $\dI_\lambda = s_\lambda + \sum_\alpha c_\alpha \YQS_\alpha$ where $\alpha \succ \lambda$ and $c_\alpha\geq 0$, if and only if $\lambda$ is a partition.

Using the decomposition given in Theorem~\ref{thm:main}, we can discuss Schur positivity for the dual immaculate quasisymmetric functions.  We note that this proposition appeared as a consequence of the work in~\cite[Corollary 3.40]{BerBerSalSerZab14}, but is proved differently.

\begin{proposition}\label{schurPosProp}
If $f= \sum_{\alpha} c_\alpha \dI_\alpha$ is symmetric and $c_\alpha\geq 0$ for all $\alpha$, then $f$ is Schur positive.
\end{proposition}
\begin{proof}
  If $f$ is symmetric, but not Schur positive, then by Theorem~\ref{SchurToYQS} $f$ would not be Young quasisymmetric Schur positive.  However, if $c_\alpha\geq 0$ for all $\alpha$, then Theorem~\ref{thm:main} implies that $f$ is Young quasisymmetric Schur positive, a contradiction.
\end{proof}

We can characterize when a single dual immaculate quasisymmetric function is symmetric.  In particular,  we show it must be indexed by a certain hook shape.

\begin{proposition}
Let $\alpha$ be a composition of $n$ of length $k+1$.  Then $\dI_\alpha$ is symmetric if and only if $\alpha=(n-k, 1^{k})$.
\end{proposition}
\begin{proof}
\noindent $(\Rightarrow)$ Suppose $\dI_\alpha$ is symmetric.  For any composition $\alpha$ of $n$ with length $k+1$, there is exactly one DIRT of shape $(1^{k},n-k)$ and row strip shape $\alpha^{rev}$.  Thus, Theorem~\ref{thm:main} implies that  $\YQS_{(1^{k},n-k)}$ appears in the decomposition of $\dI_\alpha$ with  coefficient 1.  Since $\dI_\alpha$ is symmetric it can be written in the Schur basis. Therefore Theorem~\ref{SchurToYQS} implies that all the Young quasisymmetric Schur functions whose indices rearrange $(n-k,1^{k})$ appear in the decomposition of $\dI_\alpha$ with  coefficient  at least 1.   In particular, Theorem~\ref{SchurToYQS} implies that $\YQS_{(n-k, 1^{k})}$ appears in the decomposition of $\dI_\alpha$ with nonzero coefficient.    Therefore, there must be a DIRT of shape $(n-k, 1^{k})$ and row strip shape $\alpha^{rev}$. Since $(n-k, 1^{k})$ is the  largest in dominance order for compositions of $n$ of length $k+1$, Lemma~\ref{domLem} implies that  $\alpha= (n-k, 1^{k})$ as claimed.\\

\noindent$(\Leftarrow)$ Suppose that $\alpha=(n-k, 1^{k})$.  It is not hard to check that the only DIRTs with row strip shape $(n-k, 1^{k})^{rev}$ must have shape which is a rearrangement of $(n-k, 1^{k})$.  Moreover, there is exactly one such DIRT for each rearrangement of $(n-k, 1^{k})$.  Thus,
$$
\dI_{(n-k, 1^{k})} = \sum \YQS_{\beta}
$$
where the sum is over all rearrangements of $(n-k, 1^{k})$.  Hence, Theorem~\ref{SchurToYQS} implies
$$
\dI_{(n-k, 1^{k})} = s_{(n-k, 1^{k})}
$$
and so $\dI_{(n-k, 1^{k})}$ is symmetric.
\end{proof}

Although we defined the Young quasisymmetric Schur functions as sums over fillings of semistandard Young composition tableaux, the original definition in~\cite{LMvW13} is
\begin{equation}\label{YQSEq}
\YQS_\alpha = \rho(\QS_{\alpha^{rev}}),
\end{equation}
where $\QS_{\alpha}$ is the quasisymmetric Schur function indexed by $\alpha$ and $\rho: QSym \rightarrow QSym$ is the algebra automorphism defined on the fundamental basis by
$$
\rho(F_\beta) = F_{\beta^{rev}}.
$$
The map $\rho$ is refered to as the \emph{star involution}.  A useful fact about $\rho$ is that it is the identity on $Sym$.

In~\cite{HLMvW10}, it was shown that if $s_\lambda$ is a Schur function, then the product $s_\lambda\QS_\alpha$ expands positively in the quasisymmetric Schur basis.  We now show an analogous result involving both the dual immaculate basis and the Young quasisymmetric Schur basis.

\begin{proposition}
Let $s_\lambda$ be any Schur function and $\dI_\alpha$ be any dual immaculate quasisymmetric function.  Then the product $s_\lambda \dI_\alpha$ expands positively in the Young quasisymmetric Schur basis.
\end{proposition}
\begin{proof}
First we show that if $s_\lambda$ is any Schur function and $\YQS_\alpha$ is any Young quasisymmetric Schur function, then $s_\lambda\YQS_\alpha$ expands positively in the Young quasisymmetric Schur basis.  The result will then follow by Theorem~\ref{thm:main}.

Let $\rho$ be the star involution.  Since $\rho$ is the identity on $Sym$, $\rho(s_\lambda) = s_\lambda$.  Moreover, by Equation~\ree{YQSEq} we have that $\YQS_\alpha = \rho(\QS_{\alpha^{rev}})$.  Thus,
\begin{align*}
s_\lambda\YQS_\alpha &= \rho(s_\lambda)\rho(\QS_{\alpha^{rev}})\\
&= \rho(s_\lambda\QS_{\alpha^{rev}})\\
&= \rho\left(\sum c_\beta \QS_{\beta}\right)\\
&=\sum c_\beta \rho(\QS_\beta)\\
&= \sum c_\beta \YQS_{\beta^{rev}}.
\end{align*}
By~\cite{HLMvW10}, we have that $c_\beta\geq 0$ for all $\beta$.  It follows that $s_\lambda\YQS_\alpha$ expands positively in the Young quasisymmetric Schur basis.
\end{proof}
We note that the product $s_\lambda\dI_\alpha$ does \textit{not}, in general, expand positively in the dual immaculate basis.  For example, $s_{(2,1)}\dI_{(1)}$ has negative terms in its expansion in the dual immaculate basis.  Having a positive expansion in the dual immaculate basis is not to be expected since a symmetric function which is Schur positive does not in general expand positively in the dual immaculate basis.  In fact, the reader may have noticed that $\dI_{(1)} = s_{(1)}$ and so $s_{(2,1)}\dI_{(1)}$ is symmetric and  Schur positive.

\subsection{Decompositions in $NSym$}

Recall that $QSym$ and $NSym$ are dual spaces.  The basis dual to the dual immaculate quasisymmetric functions is called the \emph{immaculate basis} and was introduced in~\cite{BerBerSalSerZab14}.  Just like the elements of $QSym$, the elements of $NSym$ are indexed by compositions and we denote the immaculate function indexed by $\alpha$ as $\I_\alpha$.  The basis dual to the Young quasisymmetric Schur functions is called the \emph{Young noncommutative Schur basis} and was introduced in~\cite{LMvW13}.  The Young noncommutative
Schur function indexed by $\alpha$ is denoted by $\dYQS_\alpha$.

The decomposition of the dual immaculate  functions into the Young quasisymmetric Schur functions also provides the decomposition between the duals of these bases in $NSym$.  If $V$ and $W$ are finite dimensional vector spaces with bases $B$ and $C$ respectively and $A$ is the change of basis matrix from $B$ to $C$, then $A^T$ is the change of basis matrix from the dual  basis $C^*$ of $W^*$ to the  dual   basis $B^*$ of $V^*$.  Thus, we obtain the dual version of Theorem~\ref{thm:main}.

\begin{theorem}\label{thm:mainDual}
The Young noncommutative Schur functions decompose into the immaculate functions in the following way:
$$
\dYQS_\alpha = \sum_{\beta} c_{\beta, \alpha} \I_\beta
$$
where $c_\beta,\alpha$ is the number of DIRTs of shape $\alpha$ and row strip shape $\beta^{rev}$.
\end{theorem}

For some choices of $\alpha$ the value of $c_{\beta,\alpha}$ is easy to determine as we see in the following corollary.

\begin{cor}
We have the following.
\begin{enumerate}
\item  Let $\alpha$ be a composition.  Then $\dYQS_\alpha = \I_\alpha$ if and only if $\alpha$ is a partition.

\item For the hook shape $(1^k, n-k)$, we have
$$
\dYQS_{(1^k, n-k)} = \sum_{\substack{\beta\vDash n\\ \ell(\beta)= k+1}} \I_\beta.
$$
\end{enumerate}
\end{cor}

\begin{proof}
(1) $(\Rightarrow)$ We prove the contrapositive.  Let $\alpha$ be a composition which is not a partition. Let $\lambda$ be the partition such that $\alpha$ rearranges to $\lambda$.  By Proposition~\ref{prop:rearrangeDIRT} we have $c_{\lambda,\alpha}\geq 1$.  Moreover, it is not hard to see that for any shape $\alpha$ there is a DIRT of shape $\alpha^{rev}$.  Thus, $c_{\alpha,\alpha}\geq 1$ and so Theorem~\ref{thm:mainDual} implies $\dYQS_\alpha \neq \I_\alpha$.

$(\Leftarrow)$ From Lemma~\ref{domLem}, we know that if $c_{\beta,\lambda} \neq 0$, then  $\beta \succeq \lambda$.  Since $\beta$ and $\lambda$ are both compositions of $n$ and have the same length, the fact that $\lambda$ is a partition forces $\beta = \lambda$.

(2)  It is not hard to see that  for any composition $\beta$ of length $k+1$, there is exactly one DIRT of shape $(1^k, n-k)$ with row strip shape $\beta^{rev}$.  The result then follows from Theorem~\ref{thm:mainDual}.
\end{proof}

\section{Remmel-Whitney-Style Algorithms}\label{sec:RWAlg}

 Since the coefficient of $\YQS_{\beta}$ is the number of DIRTs of shape $\beta$ with row strip shape $\alpha^{rev}$, we can decompose $\dI_{\alpha}$ without actually implementing the insertion algorithm.  Instead, we only need to find the the number of DIRTs of the correct shape and row strip shape.   We now explain how to find the DIRTs using an algorithm similar to the Remmel-Whitney method~\cite{RemWhi} used to multiply Schur functions.   The algorithm is recursive and produces a rooted tree where each node is a DIRT.

Suppose that we want to decompose $\dI_{\alpha}$ with $\alpha= (\alpha_1,\alpha_2,\dots, \alpha_\ell)$ into Young quasisymmetric Schur functions.   First, we set the root node to be the dual immaculate recording tableau:
$$
\tableau{1&2&3 & \cdot \cdot & \alpha_\ell}.
$$

Now  we describe how to create \emph{the children} of a node. Recall that given the composition $\alpha$ we defined
$j_m =\sum_{i=1}^{m-1} \alpha_{\ell-i-1}$.  Given a DIRT $Q$  with $k<\ell$ rows,  we create a child of $Q$ by placing the integers in $[j_k+1, j_{k+1}]$ one at a time into $Q$ using the following rules:
\begin{enumerate}
\item[(1)] The element  $j_k+1$ is placed in the leftmost column of the DIRT $Q$ below  its last row.
\item[(2)] Each subsequent element is placed at the end of a row strictly to the right of the last element placed.
\item[(3)] No element can be placed at the end of a row of length $m$ if there exists a row of length $m+1$ below this row.
\end{enumerate}
This algorithm continues until all the terminal nodes are dual immaculate recording tableaux with $\ell$ rows.

It is clear that this algorithm forces the rows to increase.  Moreover, (1) and (2) force the row strip condition and (3) forces the recording triple rule.   It is also clear that the leftmost column must increase from top to bottom.  Thus, the nodes are DIRTs with row strip shape $\alpha^{rev}$.  An inductive argument shows that every DIRT with row strip shape $\alpha^{rev}$ appears as a node in the diagram.

\begin{figure}
\begin{center}
\begin{tikzpicture}
\node(a) at (.5,0) {\tableau{1&2}};

\node(b) at (-2,-1) {\tableau{1&2\\3&4}};
\node(c) at (3,-1) {\tableau{1&2&4\\3}};
\draw(c)--(a)--(b);

\node(j) at (-4,-2.5) {\tableau{1&2\\3&4\\5&6}};
\node(d) at (-2,-2.5) {\tableau{1&2\\3&4&6\\5}};
\node(e) at (0,-2.5)  {\tableau{1&2&6\\3&4\\5}};
\draw(d)--(b)--(e);
\draw(j)--(b);

\node(f) at (2,-2.5)  {\tableau{1&2&4\\3\\5&6}};
\node(g) at (4,-2.5)  {\tableau{1&2&4\\3&6\\5}};
\node(h) at (6,-2.5)  {\tableau{1&2&4&6\\3\\5}};

\draw(f)--(c)--(g);
\draw(h)--(c);
\end{tikzpicture}\caption{The dual immaculate recording tableaux for the decomposition of $\dI_{(2,2,2)}.$} \label{RWExFig}
\end{center}
\end{figure}
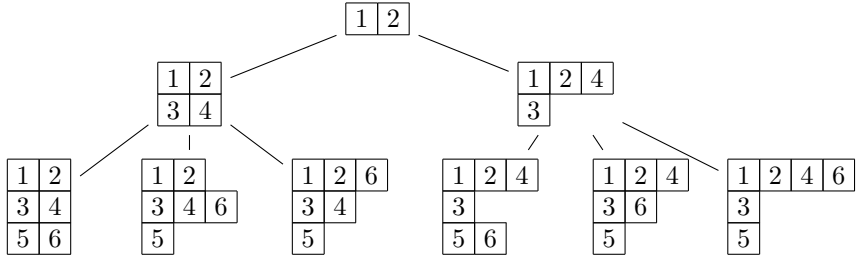

As an example of this algorithm, suppose that we want to decompose $\dI_{(2,2,2)}$.  The rooted tree in Figure~\ref{RWExFig} shows the output of the algorithm.  It follows that
$$
\dI_{(2,2,2)} = \YQS_{(2,2,2)}+\YQS_{(2,1,3)}+\YQS_{(1,3,2)}+2\YQS_{(1,2,3)}+\YQS_{(1,1,4)}.
$$
From this example, one can see that this algorithm is advantageous in  that it does not require knowing what the standard immaculate or standard Young composition tableaux are.    However, it is disadvantageous in  that it is recursive and so one must find all the smaller DIRTs in order to complete the algorithm.

We now provide a similar algorithm to find the coefficients of the decomposition of the Young noncommutative Schur functions into the immaculate functions.  We will produce a rooted tree from which we can read off the coefficients.  If we are trying to find the coefficients of $\dYQS_\alpha$, we will have a rooted tree  such that each node is a partially filled diagram with shape $\alpha$.  The root node is the empty filling.

Now we describe how to construct the children of a node.  This construction is a consequence of the definition of a DIRT.  Suppose we are at level $i$ of the rooted tree where we count the empty filling as the zero$^{th}$ level.  We place an $i$ in the empty cell in the leftmost column of the $i^{th}$ row counting from top to bottom.  Then we place $i$'s into empty cells so that the following hold. 
\begin{enumerate}
\item An $i$ can only be placed in an empty cell provided the cell immediately to its left is nonempty.
\item An $i$ cannot be placed in a column such that there is a nonempty cell below which is the rightmost nonempty cell in its row.
\item At most one $i$ can appear in each column.
\end{enumerate}

We note that the fillings we obtain in the algorithm are not necessarily DIRTs since there can be repeated entries.  However, the diagrams give the row strips from which one can easily construct the corresponding DIRTs.  Given the rooted tree, the coefficient of $\I_{(\beta_1,\beta_2,\dots, \beta_\ell)}$ in the decomposition of $\dYQS_\alpha$  is the number of fillings such that the number of $i$'s is $\beta_{\ell-i+1}$. We also note that some nodes do not have children as can be seen in the two leftmost fillings on the 2$^{nd}$ row of the rooted tree in Figure~\ref{RW2ExFig}.  It is not hard to check that this algorithm always produces all fillings which correspond to the DIRTs for the decomposition and that every filling it produces corresponds to some DIRT.

\begin{figure}

\begin{center}
\noindent\resizebox{\textwidth}{!}{\begin{tikzpicture}

\node(a) at (-2,-.75) {\tableau{{}&{}&{}\\ {}&{}\\ {}}};

\node(b) at (-7,-2.5) {\tableau{{1}&{}&{}\\ {}&{}\\ {}}};
\node(c) at (3,-2.5) {\tableau{{1}&{1}&{1}\\ {}&{}\\ {}}};
\node(d) at (-2,-2.5) {\tableau{{1}&{1}&{}\\ {}&{}\\ {}}};
\draw(c)--(a)--(b);
\draw(d)--(a);

\node(e) at (-6,-5) {\tableau{{1}&{2}&{}\\ {2}&{}\\ {}}};
\node(f) at (-4.5,-5) {\tableau{{1}&{2}&{2}\\ {2}&{}\\ {}}};
\node (z) at (-7.5,-5) {\tableau{{1}&{}&{}\\ {2}&{2}\\ {}}};
\node(y) at (-9,-5)  {\tableau{{1}&{}&{}\\ {2}&{}\\ {}}};
\draw (e)--(b)--(f);
\draw (y)--(b)--(z);

\node(g) at (-3,-5) {\tableau{{1}&{1}&{}\\ {2}&{}\\ {}}};
\node(h) at (-1.5,-5) {\tableau{{1}&{1}&{}\\ {2}&{2}\\ {}}};
\node(i) at (0,-5) {\tableau{{1}&{1}&{2}\\ {2}&{}\\ {}}};
\node(j) at (1.5,-5) {\tableau{{1}&{1}&{2}\\ {2}&{2}\\ {}}};
\node (j') at (1.6, -4.25) {};

\draw (g)--(d)--(h);
\draw (i)--(d)--(j');

\node(k) at (3,-5) {\tableau{{1}&{1}&{1}\\ {2}&{}\\ {}}};
\node(l) at (4.5,-5) {\tableau{{1}&{1}&{1}\\ {2}&{2}\\ {}}};
\draw (k)--(c)--(l);

\node(m) at (-6,-7) {\tableau{{1}&{2}&{3}\\ {2}&{3}\\ {3}}};
\node(n) at (-4.5,-7) {\tableau{{1}&{2}&{2}\\ {2}&{3}\\ {3}}};
\draw (m)--(e);
\draw (n)--(f);

\node(o) at (-3,-7) {\tableau{{1}&{1}&{3}\\ {2}&{3}\\ {3}}};
\node(p) at (-1.5,-7) {\tableau{{1}&{1}&{3}\\ {2}&{2}\\ {3}}};
\node(q) at (0,-7) {\tableau{{1}&{1}&{2}\\ {2}&{3}\\ {3}}};
\node(r) at (1.5,-7) {\tableau{{1}&{1}&{2}\\ {2}&{2}\\ {3}}};
\draw (o)--(g);
\draw (p)--(h);
\draw (q)--(i);
\draw (r)--(j);

\node(s) at (3,-7) {\tableau{{1}&{1}&{1}\\ {2}&{3}\\ {3}}};
\node(t) at (4.5,-7) {\tableau{{1}&{1}&{1}\\ {2}&{2}\\ {3}}};
\draw (s)--(k);
\draw (t)--(l);

\end{tikzpicture}}\caption{The fillings  for the decomposition of $\dYQS_{(1,2,3)}.$} \label{RW2ExFig}
\end{center}
\end{figure}
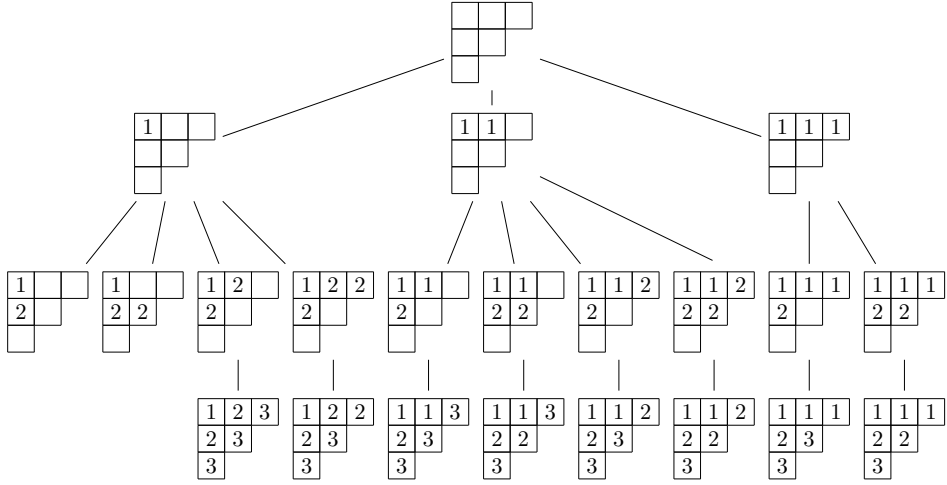

The rooted tree for $\dYQS_{(1,2,3)}$ is displayed in Figure~\ref{RW2ExFig}.    From it we see that
$$
\dYQS_{(1,2,3)} = \I_{(3,2,1)}+ \I_{(2,3,1)}+ \I_{(3,1,2)} + 2\I_{(2,2,2)}+\I_{(1,3,2)}+\I_{(2,1,3)}+\I_{(1,2,3)}.
$$

\section{Future Directions}
A natural next step is to investigate the coefficients when the Young quasisymmetric Schur functions are expanded into dual immaculate quasisymmetric functions.  If $\YQS_{\beta}$ appears in the decomposition of $\dI_{\alpha}$, then Theorem~\ref{thm:main} implies there is a DIRT of shape $\beta$ and row strip shape $\alpha^{rev}$. Lemma~\ref{domLem} implies that  $\alpha \succeq \beta$  in dominance order. Moreover, there is exactly one DIRT  of shape $\alpha$ and row strip shape $\alpha^{rev}$.  Theorem~\ref{thm:main} therefore implies that
\begin{equation}\label{yqsTodIEq}
\YQS_{\alpha} =  \dI_{\alpha} - \sum_{\beta} c_{\alpha,\beta} \YQS_{\beta},
\end{equation}
where is the sum is now over $\beta$ which are strictly  smaller than $\alpha$ in dominance order.  Since $\dI_{(1^k,n-k)} = \YQS_{(1^k,n-k)}$, equation~\ree{yqsTodIEq} along with induction implies that the Young quasisymmetric Schur functions can be decomposed into the dual immaculate quasisymmetric functions with integer coefficients.  Based on some calculations in Sage, we have the following conjectures.
\begin{conjecture}
Let $\alpha, \beta \vDash n$. If
$$
\YQS_\alpha = \sum b_{\alpha,\beta} \dI_\beta
$$
then $b_{\alpha,\beta}\in \{-1,0,1\}$.  Moreover, for a fixed $\alpha$,
$$
\sum b_{\alpha,\beta} =\begin{cases} 1 & \mbox{ if } \alpha = (1^{k}, n-k),\\ 0& \mbox{otherwise} \end{cases}
$$
for some $k$.
\end{conjecture}

\begin{conjecture}

If $\lambda= (\lambda_1,\lambda_2, \dots, \lambda_k)$ is a partition with all $k$ parts  distinct, then
$$
\YQS_\lambda = \sum_{\sigma \in S_k} (-1)^{\ell(\sigma)} \dI_{\sigma(\lambda)}
$$
where $\sigma(\lambda)  = (\lambda_{\sigma(1)},\lambda_{\sigma(2)}, \dots, \lambda_{\sigma(k)})$ and $\ell(\sigma)$ is the length of  $\sigma$ (i.e. the minimum number of transpositions of the form (i, i+1) needed to generate $\sigma$).
\end{conjecture}

In~\cite{BerBerSalSerZab14}, the authors give a formula for the number of  standard immaculate tableaux for a fixed shape. Using this formula and our bijection  we hope to find a formula for the number of standard Young composition tableaux in terms of the number of standard immaculate tableaux.  It would also be interesting to investigate the relationship between these functions and other new bases for quasisymmetric functions.  For example the shin basis~\cite{CamFelLigShuXu14} is a new basis for $NSym$ whose dual is also in $QSym$.

\section{Acknowledgements}\label{sec:ack}
The authors would like to thank Luis Serrano for helpful conversations and data.  We used the open-source software {\tt Sage} and its combinatorial features {\tt Sage-Combinat} for computer explorations.

\nocite{*}
\bibliographystyle{alpha}
\bibliography{dIFPSACbib}
\label{sec:biblio}

\end{document}